\documentclass{article}
\usepackage[utf8]{inputenc}
\usepackage{bm,amsfonts,amsthm,amsmath,amssymb}
\usepackage{graphicx,color}
\usepackage{mathrsfs}
\usepackage{indentfirst}
\usepackage{bbm}
\usepackage{appendix}

\DeclareMathAlphabet{\mathpzc}{OT1}{pzc}{m}{it}

\def\eqdefa{\buildrel\hbox{\footnotesize def}\over =}

\newcommand{\ud}{\mathrm{d}}

\newcommand{\vv}{\mathbf{v}}

\newcommand{\aaa}{\mathbf{a}}
\newcommand{\xx}{\mathbf{x}}

\newcommand{\nn}{\mathbf{n}}

\newcommand{\hh}{\mathbf{h}}

\newcommand{\sss}{\mathbf{s}}

\newcommand{\A}{\mathbf{A}}

\newcommand{\WW}{\mathbf{W}}

\newcommand{\DD}{\mathbf{D}}

\newcommand{\II}{\mathbf{I}}

\newcommand{\CD}{\mathcal{D}}

\newcommand{\CE}{\mathcal{E}}

\newcommand{\CF}{\mathcal{F}}

\newcommand{\CG}{\mathcal{G}}

\newcommand{\CA}{\mathcal{A}}

\newcommand{\CH}{\mathcal{H}}
\newcommand{\CL}{\mathcal{L}}

\newcommand{\Ff}{\mathfrak{F}}

\newcommand{\ML}{\mathscr{L}}

\newcommand{\Fp}{\mathfrak{p}}

\newcommand{\Rmnum}[1]{\uppercase\expandafter{\romannumeral #1}}

\newtheorem{theorem}{Theorem}

\newtheorem{lemma}[theorem]{Lemma}

\numberwithin{theorem}{section}
\numberwithin{equation}{section}

\allowdisplaybreaks[4]
\title{Global strong solutions to the frame hydrodynamics for biaxial nematic phases}
\author{Minjiang Feng\footnote{ School of Mathematics and Statistics, Guizhou University, Guiyang 550025, China (mjfengmath@163.com) },
Sirui Li\footnote{Corresponding author, School of Mathematics and Statistics, Guizhou University, Guiyang 550025, China (srli@gzu.edu.cn) }, 
Qi Zeng\footnote{ School of Mathematics and Statistics, Guizhou University, Guiyang 550025, China (qzengmath@163.com)}}

\date{}

\begin{document}
\maketitle
\begin{abstract}
In this article, we consider the frame hydrodynamics of biaxial nematic phases, a coupled system between the evolution of the orthonormal frame and the Navier--Stokes equation, which is derived from a molecular-theory-based dynamical tensor model about two second-order tensors. In  two and three dimensions, we establish global well-posedness of strong solutions to the Cauchy problem of frame hydrodynamics for small initial data. The key ingredient of the proof relies on estimates of nonlinear terms with rotational derivatives on $SO(3)$, together with the dissipative structure of the frame hydrodynamics.

\textbf{Keywords.} liquid crystals; biaxial nematic phase; frame hydrodynamics; well-posedness; global strong solutions

\textbf{Mathematics Subject Classification (2020).}\quad  35Q35, 35K55, 35A01, 76A15, 76D03
\end{abstract}


\section{Introduction}\label{introduction-section}

Liquid crystals exhibit local anisotropy, with physical properties between anisotropic solids and isotropic liquids. To most of us, the most familiar case is the nematic phase composed of rodlike molecules. For the uniaxial nematic phase, its local anisotropy can be described by a unit vector field $\nn\in\mathbb{S}^2$.  The hydrodynamics of the vector $\nn$ is the well-known Ericksen--Leslie model \cite{Eri1961,Les1968}. For the biaxial nematic phase, the orientational order needs to be characterized by an orthonormal frame field $\Fp\in SO(3)$, since its local anisotropy is no longer axisymmetric. The frame hydrodynamics of biaxial nematics consist of evolution equations for the orthonormal frame field coupled with the Navier--Stokes equation for the fluid velocity field.

The Ericksen--Leslie model, describing uniaxial nematic flows, has been widely studied (see \cite{LW2014,Ball2017,WZZ2021} and the references therein). We briefly review key analytical works on this model. The global existence of two-dimensional weak solutions \cite{LLW2010,HX2012,HLW2014,WW2014} and the uniqueness of such solutions \cite{LW2010,WWZ2016} were discussed, respectively. Under a special assumption of initial datum, the global well-posedness of three-dimensional weak solutions was studied \cite{LW2016}. The strong well-posedness of the general Ericksen--Leslie model was investigated for the whole space \cite{WZZ2013,WW2014} and for the bounded domain \cite{HNPS2016,Wang2024}, respectively. When inertial effects are considered, the uniaxial hydrodynamics exhibit the hyperbolic feature of the system. For the inertial Ericksen--Leslie model, the strong well-posedness \cite{JL2019,CW2020} and global regularity \cite{HJLZ2021} of the solutions for small initial data were analyzed, respectively. The almost global well-posedness for small and smooth initial data was established in \cite{HJZ2025}. For the Poiseuille flow of the inertial Ericksen--Leslie model, singularity formation and global weak solutions were also studied in \cite{CHL2020,CHX2024,HS2025}. On the other hand, the original Ericksen--Leslie model is opaque to molecular architectures, meaning that some hydrodynamic coefficients are difficult to determine. This deficiency can be addressed by establishing connections between the Ericksen--Leslie model and molecular models \cite{EZ2006,KD1983,WZZ2015-CPAM} or tensor models \cite{HLWZZ2015,LWZ2015,WZZ2015-SIAM,LW2020,LWZ2025-PYD}, respectively.

For biaxial hydrodynamics \cite{Saupe1981,Liu1981,GV1985,LLC1992}, their governing equations, which should be equivalent, can be derived directly from the symmetry of biaxial nematic phases. However, their coefficients remain phenomenological, with unclear mutual relationships. Recent works \cite{LX2023,LX2023-JDE} resolve this deficiency by deriving frame hydrodynamics from a molecular-theory-based tensor model involving two second-order tensors, via formal expansion \cite{LX2023} and rigorous biaxial limit \cite{LX2023-JDE}. The resulting frame model links coefficients to molecular parameters, inherits the energy dissipation of tensor models, and reduces to the Ericksen--Leslie model under uniaxial anisotropy.

Although biaxial hydrodynamics was proposed as early as the 1980s, analytical results remain scarce due to the inherent complexity of the model. For simplified versions of the biaxial model, several analytical results exist. The global well-posedness of two-dimensional weak solutions was established in \cite{LLW2020}, with solutions maintaining regularity except at a finite number of isolated singular times. The global solutions were studied for both incompressible \cite{GL2022,GL2025} and compressible \cite{FL2025} cases. For the compressible non-isothermal case, the global existence and large-time behavior of classical solutions to the three-dimensional exterior problem were obtained in \cite{LWZ2025}.

A good form of the model is crucial for well-posedness analysis of the solution. 
The form of the biaxial hydrodynamics presented in \cite{LX2023}, expressed by all the components of the orthonormal frame field, turns out to be convenient for analyses. Using this formulation, both local strong solutions in $\mathbb{R}^d(d=2,3)$ and global weak solutions in $\mathbb{R}^2$ have been investigated in \cite{LWX2022,LWX2024}.
To our knowledge, this is the first analytic study for the {\it full-form} biaxial hydrodynamics.

The main goal of this article is to prove the global well-posedness of strong solutions to the {\it full-form} biaxial hydrodynamics with small initial data. Our methods and results can be seen as an extension of the work initiated by Wang--Zhang--Zhang \cite{WZZ2013} for the global strong solution to the Ericksen--Leslie model. However, some new and essential difficulties appear due to the complex coupling of the system structure. The main difficulty lies in deriving the estimates of nonlinear terms with rotational derivatives on $SO(3)$ in the energy dissipative law. To facilitate higher-order energy estimates, it is necessary to rewrite the elasticity of biaxial nematic phases in an equivalent form. Furthermore, a key orthogonal decomposition, with respect to the tangential space  at a point on $SO(3)$, will be fully utilized to gain the desired dissipative estimates.

\subsection{Preliminaries}

Let us introduce some notations regarding tensors and orthonormal frames. The symbol $\otimes$ represents the tensor product. For any two $n$th-order tensors $U_1$ and $U_2$, the dot product $U_1\cdot U_2$ is defined as the sum of the products of their corresponding coordinates, i.e.,
\begin{align*}
    U_1\cdot U_2=(U_1)_{i_1\cdots i_n}(U_2)_{i_1\cdots i_n},\quad |U_1|^2=U_1\cdot U_1.
\end{align*}
Here, we adopt the Einstein summation convention for repeated indices.
We denote by 
\begin{align*}
    SO(3)\eqdefa\{\Lambda\in\mathbb{R}^{3\times3}|\Lambda^{\top}\Lambda=\II_3\}
\end{align*}
the set of noncommutative Lie groups of orthonormal frames, where $\II_3$ is the $3\times3$ identity tensor.
For any frame $\Fp=(\nn_1,\nn_2,\nn_3)\in SO(3)$, the tangential space of $SO(3)$ at a point $\Fp$ can be defined by $T_{\Fp}SO(3)$, which is spanned by a set of orthogonal bases:
\begin{align*}
  V_1=(0,\nn_3,-\nn_2),\quad V_2=(-\nn_3,0,\nn_1),\quad V_3=(\nn_2,-\nn_1,0).
\end{align*}
Then, the corresponding orthogonal complement space $(T_{\Fp}SO(3))^{\perp}$ is spanned by
\begin{align*}
  W_1=&(0,\nn_3,\nn_2),\quad W_2=(\nn_3,0,\nn_1),\quad W_3=(\nn_2,\nn_1,0),\\
  W_4=&(\nn_1,0,0),\quad
  W_5=(0,\nn_2,0),\quad
  W_6=(0,0,\nn_3).
\end{align*}

The tangent space $T_{\Fp}SO(3)$ is crucial for defining the rotational differential operator and the orthogonal decomposition. To begin with, the differential operators $\ML_k(k=1,2,3)$ on $T_{\Fp}SO(3)$ can be defined by taking the inner products of $V_k(k=1,2,3)$ and
$\partial/\partial\Fp=(\partial/\partial\nn_1,\partial/\partial\nn_2,\partial/\partial\nn_3)$, respectively,
\begin{align}\label{diff-ML1-3}
  \left\{
  \begin{aligned}
    &\ML_1\eqdefa V_1\cdot\frac{\partial}{\partial\Fp}=\nn_3\cdot\frac{\partial}{\partial\nn_2}-\nn_2\cdot\frac{\partial}{\partial\nn_3},\\
    &\ML_2\eqdefa V_2\cdot\frac{\partial}{\partial\Fp}=\nn_1\cdot\frac{\partial}{\partial\nn_3}-\nn_3\cdot\frac{\partial}{\partial\nn_1},\\
    &\ML_3\eqdefa V_3\cdot\frac{\partial}{\partial\Fp}=\nn_2\cdot\frac{\partial}{\partial\nn_1}-\nn_1\cdot\frac{\partial}{\partial\nn_2}.
  \end{aligned}
  \right.
\end{align}
Here, the differential operators $\ML_k(k=1,2,3)$ actually give the derivatives along the infinitesimal rotation about $\nn_k$. Their action on $\nn_i$ yields $\ML_k\nn_i=\epsilon^{ijk}\nn_j$, where $\epsilon^{ijk}$ is the Levi-Civita symbol.
These operators $\ML_k$ can also act on functionals by replacing
$\partial/\partial\Fp$ with the variational derivative $\delta/\delta\Fp$.

To estimate the higher-order derivative terms related to the operators $\ML_k(k=1,2,3)$, it is indispensable for the orthogonal decomposition on $T_{\Fp}SO(3)$. Specifically, for any two matrices $A,B\in\mathbb R^{3\times3}$, the orthogonal decomposition with respect to the inner product $A\cdot B$ on $T_{\Fp}SO(3)$ can be defined by
\begin{align}\label{decomposition}
A\cdot B=\sum_{k=1}^{3}\frac{1}{|V_{k}|^2}(A\cdot V_{k})(B\cdot V_k)+\sum_{k=1}^{6}\frac{1}{|W_{k}|^2}(A\cdot W_{k})(B\cdot W_{k}).     
\end{align}
Such an orthogonal decomposition (\ref{decomposition}) has played a key role in the preceding analytical works \cite{LWX2022,LWX2024}. Furthermore, for any $\Fp=(\nn_1,\nn_2,\nn_3)\in SO(3)$ and any first-order differential operators $\CD$, and for $\alpha,\beta=1,2,3$, the following simple properties hold:
\begin{align}\label{frame-differ-properties}
\left\{
\begin{aligned}
&\mathcal{D}\nn_1=(\mathcal{D}\nn_1\cdot\nn_2)\nn_2+(\mathcal{D}\nn_1\cdot\nn_3)\nn_3,\\
&\mathcal{D}\nn_2=(\mathcal{D}\nn_2\cdot\nn_1)\nn_1+(\mathcal{D}\nn_2\cdot\nn_3)\nn_3,\\
&\mathcal{D}\nn_3=(\mathcal{D}\nn_3\cdot\nn_2)\nn_2+(\mathcal{D}\nn_3\cdot\nn_1)\nn_1,\\
&\mathcal{D}\nn_{\alpha}\cdot\nn_{\beta}+\mathcal{D}\nn_{\beta}\cdot\nn_{\alpha}=\mathcal{D}(\nn_{\alpha}\cdot\nn_{\beta})=0,\\
&W_{\alpha}\cdot\mathcal{D}\Fp=0.
\end{aligned}
    \right.
\end{align}

\subsection{The full-form frame hydrodynamics}

 Similar to the Oseen--Frank elastic energy of uniaxial nematics, the form of the biaxial orientational elasticity can be determined by the mesoscopic symmetry of biaxial phases. Its local anisotropy is described by an orthonormal frame field $\Fp=(\nn_1,\nn_2,\nn_3)\in SO(3)$. The biaxial orientational elasticity takes the following form:
\begin{align}\label{elastic-energy}
\CF_{Bi}[\Fp]=\int_{\mathbb{R}^d}f_{Bi}(\Fp,\nabla\Fp)\ud\xx, 
\end{align}
where the elastic energy density $f_{Bi}$ is given by \cite{GV1984, SV1994, XZ2018}
\begin{align*}
f_{Bi}(\Fp,\nabla\Fp)=&\,\frac{1}{2}\Big(K_{1}(\nabla\cdot\nn_{1})^{2}+K_{2}(\nabla\cdot\nn_{2})^{2}+K_{3}(\nabla\cdot\nn_{3})^{2}\\
&\,+K_{4}(\nn_{1}\cdot\nabla\times\nn_{1})^{2}+K_{5}(\nn_{2}\cdot\nabla\times\nn_{2})^{2}+K_{6}(\nn_{3}\cdot\nabla\times\nn_{3})^{2}\\
&\,+K_{7}(\nn_{3}\cdot\nabla\times\nn_{1})^{2}+K_{8}(\nn_{1}\cdot\nabla\times\nn_{2})^{2}+K_{9}(\nn_{2}\cdot\nabla\times\nn_{3})^{2}\\
&\,+K_{10}(\nn_{2}\cdot\nabla\times\nn_{1})^2+K_{11}(\nn_{3}\cdot\nabla\times\nn_{2})^2+K_{12}(\nn_{1}\cdot\nabla\times\nn_{3})^2\\
&\,+\gamma_{1}\nabla\cdot[(\nn_{1}\cdot\nabla)\nn_{1}-(\nabla\cdot\nn_{1})\nn_{1}]+\gamma_{2}\nabla\cdot[(\nn_{2}\cdot\nabla)\nn_{2}-(\nabla\cdot\nn_{2})\nn_{2}]\\
&\,+\gamma_{3}\nabla\cdot[(\nn_{3}\cdot\nabla)\nn_{3}-(\nabla\cdot\nn_{3})\nn_{3}]\Big).
\end{align*}
Here, $f_{Bi}$ contains twelve bulk terms and three surface terms. 
The coefficients $K_i(i=1,\cdots,12)$ of bulk terms are all positive.
Since each surface term is a null Lagrangian, the coefficients $\gamma_i>0(i=1,2,3)$ can be chosen as needed later. While other equivalent forms of $f_{Bi}$ exist, the present formulation is indeed the most convenient for our subsequent estimates. Regarding the biaxial elasticity (\ref{elastic-energy}), the previous work \cite{LLW2020}
proved that its energy minimizers are smooth outside a closed set with vanishing one-dimensional Hausdorff measure.

To present the frame hydrodynamics, we introduce a local basis consisting of nine second-order tensors $\{\II_3, \sss_i, \aaa_j\}(i=1,\cdots,5; j=1,2,3)$, where five symmetric traceless tensors $\sss_i(i=1,\cdots,5)$ are given by
\begin{align*}
    &\sss_1=\nn_1\otimes\nn_1-\frac13\II_3,\quad \sss_2=\nn_2\otimes\nn_2-\nn_3\otimes\nn_3,\\
    &\sss_3=\frac{1}{2}(\nn_1\otimes\nn_2+\nn_2\otimes\nn_1),\quad \sss_4=\frac{1}{2}(\nn_1\otimes\nn_3+\nn_3\otimes\nn_1),\quad \sss_5=\frac{1}{2}(\nn_2\otimes\nn_3+\nn_3\otimes\nn_2), 
\end{align*}
and three asymmetric traceless tensors $\aaa_j(j=1,2,3)$ are given by
\begin{align*}
    \aaa_1=\nn_1\otimes\nn_2-\nn_2\otimes\nn_1,\quad \aaa_2=\nn_3\otimes\nn_1-\nn_1\otimes\nn_3,\quad \aaa_3=\nn_2\otimes\nn_3-\nn_3\otimes\nn_2.
\end{align*}

The frame hydrodynamics is a system coupling the evolution describing the motion of the orthonormal frame field $\Fp=\big(\nn_1(\xx,t),\nn_2(\xx,t),\nn_3(\xx,t)\big)\in SO(3)$ with the Navier--Stokes equation for the fluid velocity field $\vv=\vv(\xx,t)$. As noted in the introduction, the biaxial hydrodynamics has various equivalent forms. In this article, for analytical convenience, we adopt the following {\it full-form} frame hydrodynamics established in \cite{LX2023}:
\begin{align}
&\,\chi_1\dot{\nn}_2\cdot\nn_3-\frac{1}{2}\chi_1\WW\cdot\aaa_3-\eta_1\DD\cdot\sss_5+\ML_1\CF_{Bi}=0,\label{frame-equation-n1}\\
&\,\chi_2\dot{\nn}_3\cdot\nn_1-\frac{1}{2}\chi_2\WW\cdot\aaa_2-\eta_2\DD\cdot\sss_4+\ML_2\CF_{Bi}=0,\label{frame-equation-n2}\\
&\,\chi_3\dot{\nn}_1\cdot\nn_2-\frac{1}{2}\chi_3\WW\cdot\aaa_1-\eta_3\DD\cdot\sss_3+\ML_3\CF_{Bi}=0,\label{frame-equation-n3}\\
&\,\Fp=(\nn_1,\nn_2,\nn_3)\in SO(3),\label{frame-SO3}\\
&\,\dot{\vv}=-\nabla p+\eta\Delta\vv+\nabla\cdot\sigma+\Ff,\label{yuan-frame-equation-v}\\
&\,\nabla\cdot\vv=0,\label{yuan-imcompressible-v}
\end{align}
where the dot derivative $\dot{f}$ denotes the material derivative $(\partial_t+\vv\cdot\nabla)f$. In the equations (\ref{frame-equation-n1})--(\ref{frame-equation-n3}), $\ML_k\CF_{Bi} (k=1,2,3)$ stand for the variational derivatives along the infinitesimal rotation round $\nn_k (k=1,2,3)$. The following notations
\begin{align*}
\DD=\frac{1}{2}(\nabla\vv+(\nabla\vv)^T),\quad\WW=\frac{1}{2}(\nabla\vv-(\nabla\vv)^T)
\end{align*}
 represent the rate of strain tensor, skew-symmetric part of the strain rate, respectively. The condition (\ref{frame-SO3}) is the constraint about the frame $\Fp$. 

In the equation (\ref{yuan-frame-equation-v}), the pressure $p$ ensures the incompressiblity (\ref{yuan-imcompressible-v}), and $\eta$ is the viscous coefficient. The divergence of the viscous stress $\sigma$ is defined by $(\nabla\cdot\sigma)_i=\partial_j\sigma_{ij}$. The viscous stress $\sigma=\sigma(\Fp,\vv)$ is given by
\begin{align}
\sigma(\Fp,\vv)=&\,\beta_1(\DD\cdot\sss_1)\sss_1+\beta_0(\DD\cdot\sss_2)\sss_1+\beta_0(\DD\cdot\sss_1)\sss_2+\beta_2(\DD\cdot\sss_2)\sss_2\nonumber\\
&\,+\beta_3(\DD\cdot\sss_3)\sss_3-\eta_3\Big(\dot\nn_1\cdot\nn_2-\frac{1}{2}  \WW\cdot\aaa_1\Big)\sss_3\nonumber\\
&\,+\beta_4(\DD\cdot\sss_4)\sss_4-\eta_2\Big(\dot\nn_3\cdot\nn_1-\frac{1}{2}  \WW\cdot\aaa_2\Big)\sss_4\nonumber\\
&\,+\beta_5(\DD\cdot\sss_5)\sss_5-\eta_1\Big(\dot\nn_2\cdot\nn_3-\frac{1}{2}  \WW\cdot\aaa_3\Big)\sss_5\nonumber\\
&\,+\frac{1}{2}\eta_3(\DD\cdot\sss_3)\aaa_1-\frac{1}{2}\chi_3\Big(\dot\nn_1\cdot\nn_2-\frac{1}{2}  \WW\cdot\aaa_1\Big)\aaa_1\nonumber\\
&\,+\frac{1}{2}\eta_2(\DD\cdot\sss_4)\aaa_2-\frac{1}{2}\chi_2\Big(\dot\nn_3\cdot\nn_1-\frac{1}{2}  \WW\cdot\aaa_2\Big)\aaa_2\nonumber\\
&\,+\frac{1}{2}\eta_1(\DD\cdot\sss_5)\aaa_3-\frac{1}{2}\chi_1\Big(\dot\nn_2\cdot\nn_3-\frac{1}{2}  \WW\cdot\aaa_3\Big)\aaa_3,\label{sigma}
\end{align}
where the viscous coefficients in (\ref{sigma}), expressed as functions of molecular parameters, satisfy the following nonnegative definiteness conditions (see \cite{LX2023} for details):
\begin{align}\label{coefficient-conditions}
\left\{\begin{aligned}
&\beta_i\geq0,~i=1,\cdots,5,\quad \chi_j>0,~j=1,2,3,\quad \eta>0,\vspace{1ex}\\
&\beta^2_0\leq\beta_1\beta_2,~~\eta^2_1\leq\beta_5\chi_1,~~\eta^2_2\leq\beta_4\chi_2,~~\eta^2_3\leq\beta_3\chi_3.
\end{aligned}\right.
\end{align}
The relations in (\ref{coefficient-conditions}) ensure that the system \eqref{frame-equation-n1}--\eqref{yuan-imcompressible-v} fulfills an energy dissipation law (see Proposition $\ref{Energy dissipative law}$).
The body force $\Ff$ is defined by 
\begin{align}\label{body force}
\Ff_i=\partial_i\nn_1\cdot\nn_2\ML_3\CF_{Bi}+\partial_i\nn_3\cdot\nn_1\ML_2\CF_{Bi}+\partial_i\nn_2\cdot\nn_3\ML_1\CF_{Bi}.
\end{align}

\subsection{The main result}

In this subsection, we will state the main result of this article. To begin with, for any given $\Fp=(\nn_1,\nn_2,\nn_3)\in SO(3)$, we have
\begin{align*}
    |\nabla^k\Fp|^2=\sum^3_{i=1}|\nabla^k\nn_i|^2,~k\geq0.
\end{align*}
Next, we introduce the local well-posedness  and the blow-up criterion of strong solutions to the frame system \eqref{frame-equation-n1}--\eqref{yuan-imcompressible-v}, which have been studied in the previous work \cite{LWX2022}.

\begin{theorem}[see \cite{LWX2022}]\label{wcc}
Let $s\ge2$ be an integer. Suppose that $(\nabla\Fp^{(0)},\vv^{(0)})\in H^{2s}(\mathbb{R}^{d})\times H^{2s}(\mathbb{R}^{d})(d=2 ~or~ 3)$ is the given initial data satisfying $\nabla\cdot\vv^{(0)}=0$ and $\Fp^{(0)}=(\nn_{1}^{(0)},\nn_{2}^{(0)},\nn_{3}^{(0)})\in SO(3)$. Then, there exists $T>0$ and a unique solution $(\Fp,\vv)$ to the frame system \eqref{frame-equation-n1}--\eqref{yuan-imcompressible-v} such that 
\begin{align*}
  \nabla\Fp\in C([0,T];H^{2s}(\mathbb{R}^{d})),\quad\vv\in C([0,T];H^{2s}(\mathbb{R}^{d}))\cap L^2([0,T];H^{2s+1}(\mathbb{R}^{d})).  
\end{align*}
Let $T^{*}$ be the maximal existence time of the solution. If $T^{*}<+\infty$, then it is necessary that 
\begin{align*}
   \int_{0}^{T^{*}}(\|\nabla\times\vv(t)\|_{L^{\infty}}+\|\nabla\Fp\|^{2}_{L^{\infty}})=+\infty,\quad   \|\nabla\Fp\|^{2}_{L^{\infty}} =\sum_{i=1}^{3}\|\nabla\nn_{i}\|^{2}_{L^{\infty}}. 
\end{align*}
\end{theorem}

Based on Theorem \ref{wcc}, we will establish the global well-posedness of strong solutions to the frame system \eqref{frame-equation-n1}--\eqref{yuan-imcompressible-v} for small initial data.
\begin{theorem}\label{flz}
   Assume that the assumptions in Theorem \ref{wcc} are satisfied. There exists $\varepsilon_0>0$ such that if
    \begin{align*} \|\nabla\Fp^{(0)}\|_{H^{2s}}+\|\vv^{(0)}\|_{H^{2s}}\le \varepsilon_{0},  
    \end{align*}
    then the strong solution obtained by Theorem $\ref{wcc}$ is global in time.
\end{theorem}

Theorem \ref{flz} extends the global strong solution results for the Ericksen--Leslie model established in \cite{WZZ2013}.
To prove Theorem \ref{flz}, we will introduce a suitable energy functional to obtain the closed energy estimates. The key challenge lies in controlling higher-order derivative terms involving $\ML_k\CF_{Bi}(k=1,2,3)$. We address it by employing the orthogonal decomposition (\ref{decomposition}) to extract the following dissipative estimate (see Lemma \ref{lf}):
 \begin{align*}
       \sum_{k=1}^{3}\frac{1}{\chi_{k}}\|\ML_k\CF_{Bi}\|^{2}_{L^{2}}\geq~\frac{2\gamma^2}{\chi}\|\Delta\Fp\|_{L^{2}}^{2}+\text{lower order terms},
\end{align*}
where $\chi=\max\{\chi_1,\chi_2,\chi_3\}>0$ and $\gamma=\min\{\gamma_1,\gamma_2,\gamma_3\}>0$.

The remaining sections of this article are organized as follows. In Section 2, the biaxial orientational elasticity and the frame hydrodynamics are reformulated in an equivalent form, respectively. The basic energy dissipative law is recalled. The algebraic structures of the variational derivative with respect to the frame are presented. Also, the product estimate and commutator estimate are given. In Section 3, the global existence of strong solutions for small initial data is established. The dissipated energy estimate of the higher-order derivative term will be provided in the Appendix.

\section{Equivalent forms and some useful lemmas}

A good form of the model enables us to explore the intrinsic structure, thereby facilitating analysis.
This section will be devoted to reformulating the biaxial orientational elasticity and the frame hydrodynamics, and then presenting some useful lemmas used later. 

For any frame $\Fp=(\nn_1,\nn_2,\nn_3)\in SO(3)$, by vector algebra operations, we obtain the following identity relations:
\begin{align*}
&|\nn_1\times(\nabla\times\nn_1)|^2=(\nn_2\cdot\nabla\times\nn_1)^2+(\nn_3\cdot\nabla\times\nn_1)^2,\\
&|\nn_2\times(\nabla\times\nn_2)|^2=(\nn_1\cdot\nabla\times\nn_2)^2+(\nn_3\cdot\nabla\times\nn_2)^2,\\
&|\nn_3\times(\nabla\times\nn_3)|^2=(\nn_1\cdot\nabla\times\nn_3)^2+(\nn_2\cdot\nabla\times\nn_3)^2,\\
&|\nabla\nn_i|^2=(\nabla\cdot\nn_i)^2+(\nn_i\cdot\nabla\times\nn_i)^2+|\nn_i\times(\nabla\times\nn_i)|^2\\
&\qquad\qquad+\nabla\cdot[(\nn_i\cdot\nabla)\nn_i-(\nabla\cdot\nn_i)\nn_i],\quad i=1,2,3.
\end{align*}
Armed with the above identity relations, the energy density $f_{Bi}(\Fp,\nabla\Fp)$ in (\ref{elastic-energy}) can be reformulated as 
\begin{align}\label{new-elasitic-density}
    f_{Bi}(\Fp,\nabla\Fp)=\frac{1}{2}\sum^3_{i=1}\gamma_i|\nabla \nn_{i}|^2+W(\Fp,\nabla\Fp).
\end{align}
Here, the coefficients $\gamma_i(i=1,2,3)$ are taken as, respectively,
\begin{align}\label{gamma-i3}
\left\{\begin{aligned}
&\gamma_1=\min\{K_1,K_4,K_7,K_{10}\}>0,~ \gamma_2=\min\{K_2,K_5,K_8,K_{11}\}>0,\\
&\gamma_3=\min\{K_3,K_6,K_9,K_{12}\}>0, 
\end{aligned}\right.
\end{align}
and $W(\Fp,\nabla\Fp)$ is expressed by
\begin{align*}
W(\Fp,\nabla\Fp)=&\frac{1}{2}\Big(\sum^3_{i=1}k_i(\nabla\cdot\nn_i)^2+\sum^3_{i,j=1}k_{ij}(\nn_i\cdot\nabla\times\nn_j)^2\Big), 
\end{align*}
where the coefficients $k_i\geq 0,k_{ij}\geq0(i,j=1,2,3)$ are given by
\begin{align}\label{ki-kij}
\left\{\begin{aligned}
&k_1=K_1-\gamma_1,\quad k_2=K_2-\gamma_2,\quad
k_3=K_3-\gamma_3,\\
&k_{11}=K_4-\gamma_1,\quad
k_{22}=K_5-\gamma_2,\quad
k_{33}=K_6-\gamma_3,\\
&k_{31}=K_7-\gamma_1,\quad
k_{12}=K_8-\gamma_2,\quad
k_{23}=K_9-\gamma_3,\\
&k_{21}=K_{10}-\gamma_1,\quad
k_{32}=K_{11}-\gamma_2,\quad
k_{13}=K_{12}-\gamma_3.
\end{aligned}\right.
\end{align}

For simplicity, we define
\begin{align}\label{hh-i3-definition}
&(\hh_1,\hh_2,\hh_3)\eqdefa-\Big(\frac{\delta\CF_{Bi}}{\delta\nn_1},\frac{\delta\CF_{Bi}}{\delta\nn_2},\frac{\delta\CF_{Bi}}{\delta\nn_3}\Big)=-\frac{\delta\CF_{Bi}}{\delta\Fp}=\nabla\cdot\frac{\partial f_{Bi}}{\partial(\nabla\Fp)}-\frac{\partial f_{Bi}}{\partial\Fp}.
\end{align}
For any given frame field $\Fp=(\nn_1,\nn_2,\nn_3)\in SO(3)$, 
the algebraic structure of the variational derivative $\frac{\delta\CF_{Bi}}{\delta\Fp}$ has been introduced in \cite{LWX2022}, which is conducive to extracting higher-order derivative terms.

\begin{lemma}[algebraic structure \cite{LWX2022}]\label{h-decomposition}
For the terms $\hh_i(i=1,2,3)$, we have the following representation:
\begin{align}\label{H-representation}
\hh_i=\gamma_i\Delta\nn_i+k_i\nabla{\rm div}\nn_i
-\sum^3_{j=1}k_{ji}\nabla\times(\nabla\times\nn_i\cdot(\nn_j\otimes\nn_j))-\sum^3_{j=1}k_{ij}(\nn_i\cdot\nabla\times\nn_j)(\nabla\times\nn_j),
\end{align}
where the coefficients are expressed by \eqref{gamma-i3} and \eqref{ki-kij}. 
\end{lemma}

Applying the definitions of $\hh_i(i=1,2,3)$ in (\ref{hh-i3-definition}), the rotational derivatives $\ML_k\CF_{Bi}(k=1,2,3)$ are expressed by, respectively,
\begin{align}\label{rotational-vd-H}
    \ML_1\CF_{Bi}=\nn_2\cdot\hh_3-\nn_3\cdot\hh_2,\quad\ML_2\CF_{Bi}=\nn_3\cdot\hh_1-\nn_1\cdot\hh_3, \quad\ML_3\CF_{Bi}=\nn_1\cdot\hh_2-\nn_2\cdot\hh_1.
\end{align}

The next task is to rewrite the frame hydrodynamics (\ref{frame-equation-n1})--(\ref{yuan-imcompressible-v}) in an equivalent form. Using the property (\ref{frame-differ-properties}), the biaxial frame system (\ref{frame-equation-n1})--(\ref{yuan-imcompressible-v}) can be equivalently expressed by the equations for all coordinates of $\Fp=(\nn_1,\nn_2,\nn_3)$, where the constraint $\Fp=(\nn_1,\nn_2,\nn_3)\in SO(3)$ is automatically implied by these equations themselves. 
Thus, the frame hydrodynamics (\ref{frame-equation-n1})--(\ref{yuan-imcompressible-v}) can be reformulated as (see \cite{LWX2022,LWX2024} for details)
\begin{align}
    \dot{\nn}_{1}=&\Big(\frac{1}{2}\WW\cdot\aaa_{1}+\frac{\eta_{3}}{\chi_{3}}\DD\cdot\sss_{3}-\frac{1}{\chi_{3}}\ML_3\CF_{Bi}\Big)\nn_{2}\nonumber\\
    &-\Big(\frac{1}{2}\WW\cdot\aaa_{2}+\frac{\eta_{2}}{\chi_{2}}\DD\cdot\sss_{4}-\frac{1}{\chi_{2}}\ML_2\CF_{Bi}\Big)\nn_{3},\label{n1.}\\
    \dot{\nn}_{2}=&-\Big(\frac{1}{2}\WW\cdot\aaa_{1}+\frac{\eta_{3}}{\chi_{3}}\DD\cdot\sss_{3}-\frac{1}{\chi_{3}}\ML_3\CF_{Bi}\Big)\nn_{1}\nonumber\\
    &+\Big(\frac{1}{2}\WW\cdot\aaa_{3}+\frac{\eta_{1}}{\chi_{1}}\DD\cdot\sss_{5}-\frac{1}{\chi_{1}}\ML_1\CF_{Bi}\Big)\nn_{3},\label{n2.}\\
   \dot{\nn}_{3}=&\Big(\frac{1}{2}\WW\cdot\aaa_{2}+\frac{\eta_{2}}{\chi_{2}}\DD\cdot\sss_{4}-\frac{1}{\chi_{2}}\ML_2\CF_{Bi}\Big)\nn_{1}\nonumber\\
   &-\Big(\frac{1}{2}\WW\cdot\aaa_{3}+\frac{\eta_{1}}{\chi_{1}}\DD\cdot\sss_{5}-\frac{1}{\chi_{1}}\ML_1\CF_{Bi}\Big)\nn_{2},\label{n3.}\\
   \dot{\vv}=&-\nabla p+\eta\Delta\vv+\nabla\cdot(\sigma+\sigma^d),\label{v.}\\
\nabla\cdot\vv=&~0.\label{imcompressible-v}
\end{align}
Here, we have also rewritten the body force $\Ff$, which can be viewed as an elastic stress $\sigma^d$ (see (\ref{sigma-d}) for details). More specifically, for the body force $\Ff$ we have the following lemma (see \cite{LWX2024}).
\begin{lemma}[\cite{LWX2024}]
   For any frame $\Fp=(\nn_1,\nn_2,\nn_3)\in SO(3)$, it follows that
\begin{align*}
(\mathfrak{F})_i=\sum^3_{\alpha=1}\partial_i\nn_{\alpha}\cdot\frac{\delta\CF_{Bi}}{\delta\nn_{\alpha}}\eqdefa\partial_j\sigma^d_{ij}+\partial_i\widetilde{p},
\end{align*}
where $\widetilde{p}$ can be absorbed into the pressure term $p$ in \eqref{v.} and the elastic energy $\CF_{Bi}$ is given by \eqref{elastic-energy}, and the stress $\sigma^d_{ij}=-\frac{\partial f_{Bi}}{\partial(\partial_j\Fp)}\cdot\partial_i\Fp$.
\end{lemma}

By a direct calculation,  the stress $\sigma^d$ is expressed by
\begin{align}\label{sigma-d}
\sigma_{ij}^d(\nabla\Fp,\Fp) &=-\sum_{\alpha=1}^3\gamma_{\alpha}\partial_jn_{\alpha k}\partial_in_{\alpha k}-\sum_{\alpha=1}^3 k_{\alpha}(\nabla\cdot\nn_{\alpha})\partial_in_{\alpha j}\nonumber\\
&\quad-\sum_{\alpha,\beta=1}^3 k_{\beta\alpha}\Big((\partial_jn_{\alpha p}-\partial_pn_{\alpha j})\partial_in_{\alpha p}+n_{\beta j}n_{\beta l}(\partial_pn_{\alpha l}-\partial_ln_{\alpha p})\partial_in_{\alpha p}\nonumber\\
&\quad+n_{\beta p}n_{\beta l}(\partial_ln_{\alpha j}-\partial_jn_{\alpha l})\partial_in_{\alpha p}\Big).
\end{align}

The biaxial frame system \eqref{n1.}--\eqref{imcompressible-v} enjoys the basic energy dissipative law, which can be found in \cite{LWX2022}.
\begin{lemma}[energy dissipative law \cite{LWX2022}]\label{Energy dissipative law}
    Under the condition of $(\ref{coefficient-conditions})$, suppose that $(\Fp,\vv)$ is a strong solution to the biaxial frame system \eqref{n1.}--\eqref{imcompressible-v} with initial data $(\Fp^{(0)},\vv^{(0)})$ and $\nabla\cdot\vv^{(0)}=0$. Then it holds that 
    \begin{align*}
        &\frac{\ud}{\ud t}\Big(\frac{1}{2}\int_{\mathbb{R}^d}|\vv|^2\ud\xx+\CF_{Bi}[\Fp]\Big)
=-\eta\|\nabla\vv\|^2_{L^2}-\sum^3_{k=1}\frac{1}{\chi_k}\|\ML_k\CF_{Bi}\|^2_{L^2}\nonumber\\
&\quad-\bigg(\beta_1\|\A\cdot\sss_1\|^2_{L^2}+2\beta_0\int_{\mathbb{R}^d}(\A\cdot\sss_1)(\A\cdot\sss_2)\ud\xx+\beta_2\|\A\cdot\sss_2\|^2_{L^2}\bigg)\nonumber\\
&\quad-\Big(\beta_3-\frac{\eta^2_3}{\chi_3}\Big)\|\A\cdot\sss_3\|^2_{L^2}
-\Big(\beta_4-\frac{\eta^2_2}{\chi_2}\Big)\|\A\cdot\sss_4\|^2_{L^2}
-\Big(\beta_5-\frac{\eta^2_1}{\chi_1}\Big)\|\A\cdot\sss_5\|^2_{L^2}.
    \end{align*}
\end{lemma}
    
At the end of this section, let us present the product estimate and commutator estimate (see \cite{BCD2011,WZZ2013} for example),  which will be frequently utilized in this article.
\begin{lemma}\label{pde}
 For many multi-indices $\alpha,\beta,\gamma\in\mathbb{N}^{3}$ and any differential operator $\CD$, it follows that 
 \begin{align*}
  \|\CD^{\alpha}(fg)\|_{L^{2}}\le& ~C\sum_{|\gamma|=|\alpha|}(\|f\|_{L^{\infty}}\|\CD^{\gamma}g\|_{L^{2}}+\|g\|_{L^{\infty}}\|\CD^{\gamma}f\|_{L^{2}}),\\
  \|[\CD^{\alpha},f]\CD^{\beta}g\|_{L^{2}}\le&~ C\bigg(\sum_{|\gamma|=|\alpha|+|\beta|}\|\CD^{\gamma}f\|_{L^{2}}\|g\|_{L^{\infty}}+\sum_{|\gamma|=|\alpha|+|\beta|-1}\|\nabla f\|_{L^{\infty}}\|\CD^{\gamma}g\|_{L^{2}}\bigg).
 \end{align*}
\end{lemma}

\section{Global well-posedness of strong solutions}

This section is devoted to proving the global well-posedness of strong solutions to the biaxial frame system  \eqref{n1.}--\eqref{imcompressible-v} with small initial data. In what follows, we will use $C>0$ to denote a constant independent of the solution $(\Fp,\vv)$, but depending on the coefficients in the system (\ref{n1.})--(\ref{imcompressible-v}). The notation $\langle\cdot,\cdot\rangle$ represents the $L^{2}$-inner product in $\mathbb{R}^{d}$ with $d=2,3$.

\subsection{Energy functionals and dissipative estimates}

Assume that $(\Fp,\vv)$ is the solution to the biaxial frame system (\ref{n1.})--(\ref{imcompressible-v}) obtained by Theorem \ref{wcc} . We define two energy functionals as follows:
\begin{align}
    E_{s}(\Fp,\vv)\eqdefa&\CF_{Bi}[\Fp]+\frac{1}{2}\|\vv\|_{L^{2}}^{2}+\sum_{i=1}^{3}\CE_{i}^{s}(\Fp)+\frac{1}{2}\|\Delta^{s}\vv\|_{L^{2}}^{2},\label{E}\\  D_{s}(\Fp,\vv)\eqdefa&\eta\|\nabla\vv\|_{L^{2}}^{2}+\frac{2\gamma^{2}}{\chi}\|\Delta\Fp\|_{L^{2}}^{2}+\frac{\gamma^{2}}{4\chi}\|\Delta^{s+1}\Fp\|_{L^{2}}^{2}+\eta\|\nabla\Delta^{s}\vv\|_{L^{2}}^{2}.\label{D}  
\end{align}
Here, $s\geq2$ is an integer, $\CF_{Bi}[\Fp]$ is given by (\ref{elastic-energy}) and $\CE_{i}^{s}(\Fp)(i=1,2,3)$ are expressed by
\begin{align*}
    \CE_{i}^{s}(\Fp)=\frac{1}{2}\Big(\gamma_{i}\|\Delta^{s}\nabla\nn_{i}\|^{2}_{L^{2}}+k_{i}\|\Delta^{s}\text{div}\nn_{i}\|_{L^{2}}^{2}+\sum_{j=1}^{3}k_{ji}\|\Delta^{s}(\nabla\times\nn_{i})\cdot\nn
_{j}\|_{L^{2}}^{2}\Big).
\end{align*}

By means of Sobolev's interpolation inequality and the definitions of functionals (\ref{E}) and (\ref{D}), there exist two constants $c_{0}>0$ and $C_{0}>0$ such that  
\begin{align*}
    c_{0}(\|\nabla\Fp\|_{H^{2s}}^{2}+\|\vv\|_{H^{2s}}^{2})\leq&E_{s}(\Fp,\vv)\leq C_{0}(\|\nabla\Fp\|_{H^{2s}}^{2}+\|\vv\|_{H^{2s}}^{2}),\\
    c_{0}(\|\Delta\Fp\|_{H^{2s}}^{2}+\|\nabla\vv\|_{H^{2s}}^{2})\leq&D_{s}(\Fp,\vv)\leq C_{0}(\|\Delta\Fp\|_{H^{2s}}^{2}+\|\nabla\vv\|_{H^{2s}}^{2}).
\end{align*}

Next, we give the following dissipative estimate, which will be used to handle the higher-order derivative terms.

\begin{lemma}[dissipative estimate]\label{lf}
For any given $\Fp=(\nn_1,\nn_2,\nn_3)\in SO(3)$, there exists a constant $C>0$ such that
   \begin{align}
       \sum_{k=1}^{3}\frac{1}{\chi_{k}}\|\ML_k\CF_{Bi}\|^{2}_{L^{2}}\geq~\frac{2\gamma^2}{\chi}\|\Delta\Fp\|_{L^{2}}^{2}-CE_{s}(\Fp,\vv)D_{s}(\Fp,\vv),
   \end{align}
where $\chi=\max\{\chi_1,\chi_2,\chi_3\}$ and $\gamma=\min\{\gamma_1,\gamma_2,\gamma_3\}$, and $\ML_k(k=1,2,3)$ are the rotational differential operators on $SO(3)$.
\end{lemma}

\begin{proof}
To begin with, let us recall the orthogonal decomposition (\ref{decomposition}) formed by the tangential space $T_{\Fp}SO(3)$ and its orthogonal complement. In (\ref{decomposition}), by taking $B=\Delta\Fp$ we have
 \begin{align}\label{App}
A\cdot \Delta\Fp=\sum_{k=1}^{3}\frac{1}{|V_{k}|^2}(A\cdot V_{k})(\Delta\Fp\cdot V_k)+\sum_{k=1}^{6}\frac{1}{|W_{k}|^2}(A\cdot W_{k})(\Delta\Fp\cdot W_{k}),
 \end{align}
 for any matrix $A\in\mathbb{R}^{3\times3},$ where $V_{k}(k=1,2,3)$ and $W_{k}(k=1,\cdots,6)$ are the orthogonal bases of the tangent space $T_{\Fp}SO(3)$ and its associated orthogonal complement space, respectively. According to the definitions of $W_{k}(k=1,\cdots,6),$ we have 
 \begin{align*}
     \Delta\Fp\cdot W_{k}=\nabla\cdot
     (\nabla\Fp\cdot W_{k})-\nabla\Fp\cdot\nabla W_{k}=-\nabla\Fp\cdot\nabla W_{k},
 \end{align*}
where the property (\ref{frame-differ-properties}) has been used.
 
On the other hand, taking $A=\Delta\Fp$ in (\ref{App}), we get
\begin{align}\label{45}
    \|\Delta\Fp\|_{L^{2}}^{2}\leq\frac{1}{2}\sum_{k=1}^{3}\|\Delta\Fp\cdot V_{k}\|_{L^{2}}^{2}+C\int_{\mathbb{R}^{d}}|\nabla\Fp|^2(|\nabla^2\Fp|+|\nabla\Fp|^2)\ud\xx.
\end{align}
In the same way, if $A$ in (\ref{App}) is replaced by
\begin{align*}
    A=\nabla\cdot\frac{\partial f_{Bi}}{\partial(\nabla\Fp)}-\gamma\Delta\Fp,~~\gamma=\min\{\gamma_1,\gamma_2,\gamma_3\},
\end{align*}
then we have
\begin{align}\label{46}    \int_{\mathbb{R}^{d}}&\Big(\nabla\cdot\frac{\partial f_{Bi}}{\partial(\nabla\Fp)}-\gamma\Delta\Fp\Big)\cdot\Delta\Fp\ud\xx\nonumber\\
    &\leq\frac{1}{2}\int_{\mathbb{R}^{d}}\sum_{k=1}^{3}\Big[
\Big(\nabla\cdot\frac{\partial f_{Bi}}{\partial(\nabla\Fp)}-\gamma\Delta\Fp\Big)\cdot V_{k}\Big](\Delta\Fp\cdot V_{k})\ud\xx\nonumber\\
&\quad+C\int_{\mathbb{R}^{d}}|\nabla\Fp|^{2}(|\nabla^{2}\Fp|+|\nabla\Fp|^{2})\ud\xx.
\end{align}
Furthermore, applying Lemma $\ref{h-decomposition}$ and integration by parts, we infer that
\begin{align}\label{47}
\int_{\mathbb{R}^{d}}&\Big(\nabla\cdot\frac{\partial f_{Bi}}{\partial(\nabla\Fp)}-\gamma\Delta\Fp\Big)\cdot\Delta\Fp\ud\xx\nonumber\\
= &\int_{\mathbb{R}^{d}}\Big(\nabla\cdot\frac{\partial f_{Bi}}{\partial(\nabla\Fp)}-\frac{\partial f_{Bi}}{\partial\Fp}-\gamma\Delta\Fp+\frac{\partial f_{Bi}}{\partial\Fp}\Big)\cdot\Delta\Fp\ud\xx\nonumber\\
\geq&\sum_{i=1}^{3}\int_{\mathbb{R}^{d}}\big(k_{i}|\nabla\text{div}\nn_{i}|^{2}+\sum_{j=1}^{3}k_{ji}|\nabla(\nn_{j}\cdot(\nabla\times\nn_{i}))|^{2}\big)\ud\xx\nonumber\\
&-C\int_{\mathbb{R}^{d}}|\nabla\Fp|^{2}(|\nabla^{2}\Fp|+|\nabla\Fp|^{2})\ud\xx.
\end{align}

We set $\chi=\max\{\chi_1,\chi_2,\chi_3\}$, and note that
\begin{align*}
(\hh_1,\hh_2,\hh_3)=-\Big(\frac{\delta\CF_{Bi}}{\delta\nn_1},\frac{\delta\CF_{Bi}}{\delta\nn_2},\frac{\delta\CF_{Bi}}{\delta\nn_3}\Big)=-\frac{\delta\CF_{Bi}}{\delta\Fp}=\nabla\cdot\frac{\partial f_{Bi}}{\partial(\nabla\Fp)}-\frac{\partial f_{Bi}}{\partial\Fp},
\end{align*}
where $\frac{\partial f_{Bi}}{\partial\Fp}$ are several lower-order derivative terms. Consequently, by means of the definitions of $\ML_{k}(k=1,2,3)$ and (\ref{45})--(\ref{47}), we deduce that 
\begin{align}\label{ML-estimate}
&\sum_{k=1}^{3}\frac{1}{\chi_{k}}\|\ML_k\CF_{Bi}\|^{2}_{L^{2}}=\sum_{k=1}^{3}\frac{1}{\chi_{k}}\int_{\mathbb{R}^{d}}\Big(V_{k}\cdot\frac{\delta\CF_{Bi}}{\delta\Fp}\Big)^{2}\ud\xx\nonumber\\
&\quad\geq\frac{1}{\chi}\int_{\mathbb{R}^{d}}\sum_{k=1}^{3}\Big(V_{k}\cdot\Big(\nabla\cdot\frac{\partial f_{Bi}}{\partial(\nabla\Fp)}\Big)\Big)^{2}\ud\xx-C\int_{\mathbb{R}^{d}}|\nabla\Fp|^{2}(|\nabla^{2}\Fp|+|\nabla\Fp|^{2})\ud\xx\nonumber\\
&\quad\geq\frac{2\gamma}{\chi}\int_{\mathbb{R}^{d}}\sum_{k=1}^{3}\Big[\Big(\nabla\cdot\frac{\partial f_{Bi}}{\partial(\nabla\Fp)}-\gamma\Delta\Fp\Big)\cdot V_{k}\Big](\Delta\Fp\cdot V_{k})\ud\xx\nonumber\\
&\qquad+\frac{\gamma^{2}}{\chi}\int_{\mathbb{R}^{d}}\sum_{k=1}^{3}(\Delta\Fp\cdot V_{k})^{2}\ud\xx-C\int_{\mathbb{R}^{d}}|\nabla\Fp|^{2}(|\nabla^{2}\Fp|+|\nabla\Fp|^{2})\ud\xx\nonumber\\
&\quad\geq\frac{4\gamma}{\chi}\sum_{i=1}^{3}\int_{\mathbb{R}^{d}}\Big(k_{i}|\nabla\text{div}\nn_{i}|^{2}+\sum_{j=1}^{3}k_{ji}|\nabla(\nn_{j}\cdot(\nabla\times\nn_{i}))|^{2}\Big)\ud\xx\nonumber\\
&\qquad+\frac{2\gamma^{2}}{\chi}\|\Delta\Fp\|^2_{L^{2}}-C\int_{\mathbb{R}^{d}}|\nabla\Fp|^{2}(|\nabla^{2}\Fp|+|\nabla\Fp|^{2})\ud\xx.
\end{align}
By the definitions of two functionals $E_s(\Fp,\vv)$ and $D_s(\Fp,\vv)$ in (\ref{E})--(\ref{D}), it follows that
\begin{align*}
\int_{\mathbb{R}^{d}}|\nabla\Fp|^{4}\ud\xx\leq&C\int_{\mathbb{R}^{d}}|\nabla^{2}\Fp|^{2}\ud\xx\leq CE_s(\Fp,\vv)D_s(\Fp,\vv),\nonumber\\
\int_{\mathbb{R}^{d}}|\nabla\Fp|^{2}|\nabla^{2}\Fp|\ud\xx\leq&C\int_{\mathbb{R}^{d}}|\nabla^{2}\Fp|^{2}\ud\xx\leq CE_s(\Fp,\vv)D_s(\Fp,\vv),
\end{align*}
where the fact $|\nabla\Fp|^{2}=-\Fp\cdot\Delta\Fp\leq C|\nabla^{2}\Fp|$ has been used.
Combining (\ref{ML-estimate}) with the above two estimates, we immediately obtain
\begin{align*}
 \sum_{k=1}^{3}\frac{1}{\chi_{k}}\|\ML_k\CF_{Bi}\|^{2}_{L^{2}}\geq~\frac{2\gamma^2}{\chi}\|\Delta\Fp\|_{L^{2}}^{2}-CE_{s}(\Fp,\vv)D_{s}(\Fp,\vv). 
\end{align*}
Thus, the proof of the lemma is completed.
\end{proof}

Now, we first deal with the estimates of lower order terms in (\ref{E}). From Lemma \ref{Energy dissipative law}, the dissipative inequality of the system (\ref{n1.})--(\ref{imcompressible-v}) is given by
\begin{align*}
\frac{\ud}{\ud t}\Big(\frac{1}{2}\int_{\mathbb{R}^{d}}|\vv|^{2}\ud\xx+\CF_{Bi}[\Fp]\Big)+\eta\|\nabla\vv\|_{L^{2}}^{2}+ \sum_{k=1}^{3}\frac{1}{\chi_{k}}\|\ML_k\CF_{Bi}\|^{2}_{L^{2}}\leq 0,
\end{align*}
which together with Lemma \ref{lf} implies that
\begin{align}\label{dv..}
&\frac{\ud}{\ud t}\Big(\frac{1}{2}\int_{\mathbb{R}^{d}}|\vv|^{2}\ud\xx+\CF_{Bi}[\Fp]\Big)+\eta\|\nabla\vv\|_{L^{2}}^{2}+\frac{2\gamma^{2}}{\chi}\|\Delta\Fp\|_{L^{2}}^{2}\leq CE_{s}(\Fp,\vv)D_{s}(\Fp,\vv).
\end{align}

\subsection{Estimates of higher order terms for $\vv$}
We are now in a position to estimate the higher order derivatives for $\vv$ in (\ref{E}). Acting the differential operator $\Delta^{s}$ on the equation (\ref{v.}) and taking the $L^2$-inner product with $\Delta^{s}\vv$, and noting $\nabla\cdot\vv=0$, we deduce that 
\begin{align}\label{dv}
    \frac{1}{2}\frac{\ud}{\ud t}\|\Delta^{s}\vv\|_{L^{2}}^{2}+\eta\|\nabla\Delta^{s}\vv\|_{L^{2}}^{2}=&-\langle \Delta^{s}(\vv\cdot\nabla\vv),\Delta^{s}\vv\rangle-\langle\Delta^{s}\sigma,\Delta^{s}\nabla\vv\rangle-\langle\Delta^{s}\sigma^d,\Delta^{s}\nabla\vv\rangle\nonumber\\ 
    \eqdefa&\Rmnum{1}+\Rmnum{2}+\Rmnum{3}.
\end{align}
We deal with $(\ref{dv})$ term by term. For the term $I$, it follows from Lemma \ref{pde} that 
\begin{align}\label{I-estimate}
    \Rmnum{1}&=-\langle[\Delta^{s},\vv\cdot]\nabla\vv,\Delta^{s}\vv\rangle\leq\|[\Delta^{s},\vv\cdot]\nabla\vv\|_{L^{2}}\|\Delta^{s}\vv\|_{L^{2}}\nonumber\\
    &\leq C\Big(\|\Delta^{s}\vv\|_{L^{2}}\|\nabla\vv\|_{L^{\infty}}+\|\nabla\vv\|_{L^{\infty}}\|\nabla\Delta^{s-1}\nabla\vv\|_{L^{2}}\Big)\|\Delta^{s}\vv\|_{L^{2}}\nonumber\\
    &\leq CE^{\frac{1}{2}}_{s}(\Fp,\vv)D_{s}(\Fp,\vv).
\end{align}
Recalling the definition of the stress $\sigma$, the term $II$ in (\ref{dv}) can be rewritten as
\begin{align}\label{rm2}
    \Rmnum{2}\eqdefa\Rmnum{2}_{1}+\Rmnum{2}_{2}+\Rmnum{2}_{3}+\Rmnum{2}_{4},
\end{align}
where $II_i(i=1,\cdots,4)$ are expressed by, respectively,
\begin{align*}
    \Rmnum{2}_{1}=&-\Big\langle\beta_{1}\Delta^{s}((\DD\cdot\sss_{1})\sss_{1})+\beta_{0}\Delta^{s}((\DD\cdot\sss_{1})\sss_{2})+\beta_{0}\Delta^{s}((\DD\cdot\sss_{2})\sss_{1})+\beta_{2}\Delta^{s}((\DD\cdot\sss_{2})\sss_{2}),\Delta^{s}\DD\Big\rangle,\\
    \Rmnum{2}_{2}=&-\Big\langle\Big(\beta_{3}-\frac{\eta_{3}^{2}}{\chi_{3}}\Big)\Delta^{s}((\DD\cdot\sss_{3})\sss_{3})+\Big(\beta_{4}-\frac{\eta_{2}^{2}}{\chi_{2}}\Big)\Delta^{s}((\DD\cdot\sss_{4})\sss_{4})+\Big(\beta_{5}-\frac{\eta_{1}^{2}}{\chi_{1}}\Big)\Delta^{s}((\DD\cdot\sss_{5})\sss_{5}),\Delta^{s}\DD\Big\rangle,\\
    \Rmnum{2}_{3}=&-\Big\langle\frac{\eta_{3}}{\chi_{3}}\Delta^{s}((\ML_3\CF_{Bi})\sss_{3})+\frac{\eta_{2}}{\chi_{2}}\Delta^{s}((\ML_2\CF_{Bi})\sss_{4})+\frac{\eta_{1}}{\chi_{1}}\Delta^{s}((\ML_1\CF_{Bi})\sss_{5}),\Delta^{s}\DD\Big\rangle,\\
    \Rmnum{2}_{4}=&-\frac{1}{2}\Big\langle\Delta^{s}\big((\ML_3\CF_{Bi})\aaa_{1}+(\ML_2\CF_{Bi})\aaa_{2}+(\ML_1\CF_{Bi})\aaa_{3}\big),\Delta^{s}\WW\Big\rangle.
\end{align*}
By the symmetry of $\Rmnum{2}_{1}+\Rmnum{2}_{2},$ the relation $\beta^2_0\leq\beta_1\beta_2$ in (\ref{coefficient-conditions}) and Lemma \ref{pde}, we infer that 
\begin{align}\label{II1+II2}
  \Rmnum{2}_{1}+\Rmnum{2}_{2}
  \leq&\underset{\leq0}{\underbrace{-\Big(\sum_{j=1}^{2}\beta_{j}\|\Delta^{s}\DD\cdot\sss_{j}\|_{L^{2}}^{2}+2\beta_{0}\int_{\mathbb{R}^{d}}(\Delta^{s}\DD\cdot\sss_{1})(\Delta^{s}\DD\cdot\sss_{2})\ud\xx\Big)}}\nonumber\\
 &-\Big(\beta_{3}-\frac{\eta_{3}^{2}}{\chi_{3}}\Big)\|\Delta^{s}\DD\cdot\sss_{3}\|_{L^{2}}^{2}-\Big(\beta_{4}-\frac{\eta_{2}^{2}}{\chi_{2}}\Big)\|\Delta^{s}\DD\cdot\sss_{4}\|_{L^{2}}^{2}-\Big(\beta_{5}-\frac{\eta_{1}^{2}}{\chi_{1}}\Big)\|\Delta^{s}\DD\cdot\sss_{5}\|_{L^{2}}^{2}\nonumber\\
&+C\underset{i=j~\text{if}~i,j\geq3}{\sum_{i,j=1;}^{5}} \Big(\|[\Delta^{s},\sss_{i}\cdot](\DD\cdot\sss_{j})\|_{L^{2}}+\|\sss_{i}\|_{L^{\infty}}\|[\Delta^{s},\sss_{j}\cdot]\DD\|_{L^{2}}\Big)\|\Delta^{s}\nabla\vv\|_{L^{2}}.
\end{align}
We next handle the last term in (\ref{II1+II2}). First, we have $\|\sss_i\|_{L^{\infty}}\leq C$ due to the expressions of $\sss_i$ and $\Fp=(\nn_1,\nn_2,\nn_3)\in SO(3)$. For the term $\|[\Delta^{s},\sss_{j}\cdot]\DD\|_{L^{2}}$, from Lemma \ref{pde} and the Sobolev embedding inequality, it can be estimated as
\begin{align*}
\|[\Delta^{s},\sss_{j}\cdot]\DD\|_{L^{2}}\leq C\big(\|\Delta^{s}\sss_{j}\|_{L^{2}}\|\nabla\vv\|_{L^{\infty}}+\|\nabla\sss_{j}\|_{L^{\infty}}\|\Delta^{s}\vv\|_{L^{2}}\big),
\end{align*}
where $\|\Delta^s\sss_j\|_{L^2}$ is controlled by
\begin{align*}
\|\Delta^{s}\sss_{j}\|_{L^{2}}\leq&C\|\Delta^{s}(\nn_{i}\otimes \nn_{k})\|_{L^{2}}
\leq C\big(\|[\Delta^{s},\nn_{i}\otimes]\nn_{k}\|_{L^{2}}+\|\nn_{i}\otimes\Delta^{s}\nn_{k}\|_{L^{2}}\big)\\
\leq&C\big(\|\Delta^{s}\nn_{i}\|_{L^{2}}\|\nn_{k}\|_{L^{\infty}}+\|\nabla\nn_{i}\|_{L^{\infty}}\|\nabla\Delta^{s-1}\nn_{k}\|_{L^{2}}+\|\Delta^{s}\nn_{k}\|_{L^{2}}\big)\\
\leq&C\big(\|\nabla\Fp\|_{L^{\infty}}\|\nabla\Delta^{s-1}\Fp\|_{L^{2}}+\|\Delta^{s}\Fp\|_{L^{2}}\big).
\end{align*}
Thus we have
\begin{align}\label{Delta-sD}
\|[\Delta^{s},\sss_{j}\cdot]\DD\|_{L^{2}}\leq&C\Big(\|\nabla\Fp\|_{L^{\infty}}\|\nabla\Delta^{s-1}\Fp\|_{L^{2}}\|\nabla\vv\|_{L^{\infty}}+\|\Delta^{s}\Fp\|_{L^{2}}\|\nabla\vv\|_{L^{\infty}}+\|\nabla\Fp\|_{L^{\infty}}\|\Delta^{s}\vv\|_{L^{2}}\Big)\nonumber\\
\leq&C\big(E^{\frac{1}{2}}_{s}(\Fp,\vv)+E_{s}(\Fp,\vv)\big)D^{\frac{1}{2}}_{s}(\Fp,\vv).
\end{align}
For the term $\|[\Delta^{s},\sss_{i}\cdot](\DD\cdot\sss_{j})\|_{L^{2}}$, from Lemma \ref{pde} we can estimate it as 
\begin{align*}
\|[\Delta^{s},\sss_{i}\cdot](\DD\cdot\sss_{j})\|_{L^{2}}\leq&C\big(\|\Delta^{s}\sss_{i}\|_{L^{2}}\|\DD\cdot\sss_{j}\|_{L^{\infty}}+\|\nabla\sss_{i}\|_{L^{\infty}}\|\nabla\Delta^{s-1}(\DD\cdot\sss_{j})\|_{L^{2}}\big)\\
\leq&C\big(\|\nabla\Fp\|_{L^{\infty}}\|\nabla\Delta^{s-1}\Fp\|_{L^{2}}\|\nabla\vv\|_{L^{\infty}}+\|\nabla\vv\|_{L^{\infty}}\|\Delta^{s}\Fp\|_{L^{2}}\\
&+\|\nabla\Fp\|_{L^{\infty}}\|\nabla\Delta^{s-1}(\DD\cdot\sss_{j})\|_{L^{2}}\big),
\end{align*}
where $\|\nabla\Delta^{s-1}(\DD\cdot\sss_{j})\|_{L^{2}}$ can be handled as
\begin{align*}
\|\nabla\Delta^{s-1}(\DD\cdot\sss_{j})\|_{L^{2}}\leq& C\big(\|[\nabla\Delta^{s-1},\DD\cdot]\sss_{j}\|_{L^{2}}+\|\DD\cdot\nabla\Delta^{s-1}\sss_{j}\|_{L^{2}}\big)\\
\leq&C\Big(\|\nabla\Delta^{s-1}\DD\|_{L^{2}}\|\sss_{j}\|_{L^{\infty}}+\|\nabla\DD\|_{L^{\infty}}\|\Delta^{s-1}\sss_{j}\|_{L^{2}}+\|\nabla\vv\|_{L^{\infty}}\|\nabla\Delta^{s-1}\sss_{j}\|_{L^{2}}\Big)\\
\leq&C\Big(\|\Delta^{s}\vv\|_{L^{2}}+\|\nabla^{2}\vv\|_{L^{\infty}}\big(\|\nabla\Fp\|_{L^{\infty}}\|\nabla\Delta^{s-2}\Fp\|_{L^{2}}+\|\Delta^{s-1}\Fp\|_{L^{2}}\big)\\
&+\|\nabla\vv\|_{L^{\infty}}\big(\|\nabla\Fp\|_{L^{\infty}}\|\Delta^{s-1}\Fp\|_{L^{2}}+\|\nabla\Delta^{s-1}\Fp\|_{L^{2}}\big)\Big).
\end{align*}
Then, we obtain
\begin{align}\label{Delta-sssi}
\|[\Delta^{s},\sss_{i}\cdot](\DD\cdot\sss_{j})\|_{L^{2}}\leq&C\Big(\|\nabla\vv\|_{L^{\infty}}\|\nabla\Fp\|_{L^{\infty}}\|\nabla\Delta^{s-1}\Fp\|_{L^{2}}+\|\nabla\vv\|_{L^{\infty}}\|\Delta^{s}\Fp\|_{L^{2}}+\|\nabla\Fp\|_{L^{\infty}}\|\Delta^{s}\vv\|_{L^{2}}\nonumber\\
&+\|\nabla^{2}\vv\|_{L^{\infty}}\|\nabla\Fp\|_{L^{\infty}}^{2}\|\nabla\Delta^{s-2}\Fp\|_{L^{2}}+\|\nabla^{2}\vv\|_{L^{\infty}}\|\nabla\Fp\|_{L^{\infty}}\|\Delta^{s-1}\Fp\|_{L^{2}}\nonumber\\
&+\|\nabla\vv\|_{L^{\infty}}\|\nabla\Fp\|_{L^{\infty}}^{2}\|\Delta^{s-1}\Fp\|_{L^{2}}+\|\nabla\vv\|_{L^{\infty}}\|\nabla\Fp\|_{L^{\infty}}\|\nabla\Delta^{s-1}\Fp\|_{L^{2}}\Big)\nonumber\\
\leq&C\sum^3_{k=1}E^{\frac{k}{2}}_{s}(\Fp,\vv)D^{\frac{1}{2}}_{s}(\Fp,\vv).
\end{align}
Hence, combining (\ref{II1+II2}) with the estimates (\ref{Delta-sD}) and (\ref{Delta-sssi}) leads to
\begin{align}\label{II1-2final}
    \Rmnum{2}_{1}+\Rmnum{2}_{2}\leq&C\sum^3_{k=1}E^{\frac{k}{2}}_s(\Fp,\vv)D_{s}(\Fp,\vv).
\end{align}

To estimate the term $II_3+II_4$, we introduce the following notations:
\begin{align}\label{CH}
\left\{\begin{aligned}
    \CH^{\Delta^{s}}_{1}\eqdefa~&\nn_{2}\cdot\Delta^{s}\hh_{3}-\nn_{3}\cdot\Delta^{s}\hh_{2},\\
    \CH^{\Delta^{s}}_{2}\eqdefa~&\nn_{3}\cdot\Delta^{s}\hh_{1}-\nn_{1}\cdot\Delta^{s}\hh_{3},\\
    \CH^{\Delta^{s}}_{3}\eqdefa~&\nn_{1}\cdot\Delta^{s}\hh_{2}-\nn_{2}\cdot\Delta^{s}\hh_{1}.
    \end{aligned}\right.
\end{align}
Similar arguments can be applied to control the term $II_3+II_4$ in (\ref{rm2}). By reserving higher-order derivative terms and applying Lemma \ref{pde}, one can get that
\begin{align*}
\Rmnum{2}_{3}+\Rmnum{2}_{4}=&-\Big\langle\frac{\eta_{3}}{\chi_{3}}(\CH^{\Delta^{s}}_{3})\sss_{3}+\frac{\eta_{2}}{\chi_{2}}(\CH^{\Delta^{s}}_{2})\sss_{4}+\frac{\eta_{1}}{\chi_{1}}(\CH^{\Delta^{s}}_{1})\sss_{5},\Delta^{s}\DD\Big\rangle\\
&-\frac{1}{2}\Big\langle(\CH^{\Delta^{s}}_{3})\aaa_{1}+(\CH^{\Delta^{s}}_{2})\aaa_{2}+(\CH^{\Delta^{s}}_{1})\aaa_{3},\Delta^{s}\WW\Big\rangle\\
&-\Big\langle\frac{\eta_{3}}{\chi_{3}}([\Delta^{s},\nn_{1}\cdot]\hh_{2}-[\Delta^{s},\nn_{2}\cdot]\hh_{1})\sss_{3}+\frac{\eta_{2}}{\chi_{2}}([\Delta^{s},\nn_{3}\cdot]\hh_{1}-[\Delta^{s},\nn_{1}\cdot]\hh_{3})\sss_{4}\\
&\quad+\frac{\eta_{1}}{\chi_{1}}([\Delta^{s},\nn_{2}\cdot]\hh_{3}-[\Delta^{s},\nn_{3}\cdot]\hh_{2})\sss_{5},\Delta^{s}\DD\Big\rangle\\
&-\frac{1}{2}\Big\langle([\Delta^{s},\nn_{1}\cdot]\hh_{2}-[\Delta^{s},\nn_{2}\cdot]\hh_{1})\aaa_{1}+([\Delta^{s},\nn_{3}\cdot]\hh_{1}-[\Delta^{s},\nn_{1}\cdot]\hh_{3})\aaa_{2}\\
&\quad+([\Delta^{s},\nn_{2}\cdot]\hh_{3}-[\Delta^{s},\nn_{3}\cdot]\hh_{2})\aaa_{3},\Delta^{s}\WW\Big\rangle\\
\leq&-\Big\langle\frac{\eta_{3}}{\chi_{3}}(\CH^{\Delta^{s}}_{3})\sss_{3}+\frac{\eta_{2}}{\chi_{2}}(\CH^{\Delta^{s}}_{2})\sss_{4}+\frac{\eta_{1}}{\chi_{1}}(\CH^{\Delta^{s}}_{1})\sss_{5},\Delta^{s}\DD\Big\rangle\\
&-\frac{1}{2}\Big\langle(\CH^{\Delta^{s}}_{3})\aaa_{1}+(\CH^{\Delta^{s}}_{2})\aaa_{2}+(\CH^{\Delta^{s}}_{1})\aaa_{3},\Delta^{s}\WW\Big\rangle\\
&+C\sum^3_{i,j=1;i\neq j}\|[\Delta^s, \nn_i\cdot]\hh_j\|_{L^2}\|\Delta^s\nabla\vv\|_{L^2},
\end{align*}
where $\|[\Delta^s, \nn_i\cdot]\hh_j\|_{L^2}$ can be estimated as
\begin{align}\label{Delta-s-nni-hhj}
\|[\Delta^{s},\nn_{i}\cdot]\hh_{j}\|_{L^{2}}\leq&C\big(\|\Delta^{s}\nn_{i}\|_{L^{2}}\|\hh_{j}\|_{L^{\infty}}+\|\nabla\nn_{i}\|_{L^{\infty}}\|\Delta^{s-1}\nabla\hh_{j}\|_{L^{2}}\big)\nonumber\\
\leq&C\sum^3_{k=1}E^{\frac{k}{2}}_{s}(\Fp,\vv)D^{\frac{1}{2}}_{s}(\Fp,\vv).
\end{align}
Here, it should be noted that the following estimates have been employed in (\ref{Delta-s-nni-hhj}):
\begin{align*}
    \|\hh_{j}\|_{L^{\infty}}\leq& C\big(\|\Delta\Fp\|_{L^{\infty}}+\|\nabla\Fp\|^{2}_{L^{\infty}}\big)
    \leq C\sum^2_{k=1}E^{\frac{k}{2}}_{s}(\Fp,\vv),\\
    \|\Delta^{s-1}\nabla\hh_{j}\|_{L^{\infty}}\leq&C\Big(\|\Delta^{s}\nabla\Fp\|_{L^{2}}+\|\nabla^{2}\Fp\|_{L^{\infty}}\|\Delta^{s-1}\nabla\Fp\|_{L^{2}}\\
    &+\|\nabla\Fp\|_{L^{\infty}}\|\Delta^{s}\Fp\|_{L^{2}}+\|\nabla\Fp\|^{2}_{L^{\infty}}\|\Delta^{s-1}\nabla\Fp\|_{L^{2}}\Big)\\
    \leq&C\sum^3_{k=1}E^{\frac{k}{2}}_{s}(\Fp,\vv).
\end{align*}
Consequently, combining the above inequalities yields
\begin{align}\label{II3+II4}
 \Rmnum{2}_{3}+\Rmnum{2}_{4}\leq&-\Big\langle\frac{\eta_{3}}{\chi_{3}}(\CH^{\Delta^{s}}_{3})\sss_{3}+\frac{\eta_{2}}{\chi_{2}}(\CH^{\Delta^{s}}_{2})\sss_{4}+\frac{\eta_{1}}{\chi_{1}}(\CH^{\Delta^{s}}_{1})\sss_{5},\Delta^{s}\DD\Big\rangle\nonumber\\
&-\frac{1}{2}\Big\langle(\CH^{\Delta^{s}}_{3})\aaa_{1}+(\CH^{\Delta^{s}}_{2})\aaa_{2}+(\CH^{\Delta^{s}}_{1})\aaa_{3},\Delta^{s}\WW\Big\rangle\nonumber\\
&+C\sum^3_{k=1}E^{\frac{k}{2}}_{s}(\Fp,\vv)D_{s}(\Fp,\vv).
\end{align}
It follows from (\ref{II1-2final}) and (\ref{II3+II4}) that
\begin{align}\label{II-expression-f}
  \Rmnum{2}\leq&-\Big\langle\frac{\eta_{3}}{\chi_{3}}(\CH^{\Delta^{s}}_{3})\sss_{3}+\frac{\eta_{2}}{\chi_{2}}(\CH^{\Delta^{s}}_{2})\sss_{4}+\frac{\eta_{1}}{\chi_{1}}(\CH^{\Delta^{s}}_{1})\sss_{5},\Delta^{s}\DD\Big\rangle\nonumber\\
&-\frac{1}{2}\Big\langle(\CH^{\Delta^{s}}_{3})\aaa_{1}+(\CH^{\Delta^{s}}_{2})\aaa_{2}+(\CH^{\Delta^{s}}_{1})\aaa_{3},\Delta^{s}\WW\Big\rangle\nonumber\\
&+C\sum^3_{k=1}E^{\frac{k}{2}}_{s}(\Fp,\vv)D_{s}(\Fp,\vv).  
\end{align}

It remains to control the term $\Rmnum{3}$.  Using the definition of the stress $\sigma^d$ in (\ref{sigma-d}), one can obtain
\begin{align}\label{III-expression}
\Rmnum{3}\leq&\|\Delta^s\sigma^{d}\|_{L^{2}}\|\Delta^{s}\nabla\vv\|_{L^{2}}\nonumber\\
\leq&C\bigg(\sum^3_{\alpha=1}\|\Delta^{s}(\partial_jn_{\alpha k}\partial_in_{\alpha k})\|_{L^{2}}+\sum^3_{\alpha=1}\|\Delta^{s}((\nabla\cdot\nn_{\alpha})\partial_in_{\alpha j})\|_{L^{2}}\nonumber\\
&+\sum^3_{\alpha,\beta=1}\Big(\|\Delta^{s}((\partial_jn_{\alpha p}-\partial_pn_{\alpha j})\partial_in_{\alpha p})\|_{L^{2}}\nonumber\\
&\quad+\|\Delta^{s}\{n_{\beta j}n_{\beta l}(\partial_pn_{\alpha l}-\partial_ln_{\alpha p})\partial_in_{\alpha p}\}\|_{L^{2}}\nonumber\\
&\quad+\|\Delta^{s}\{n_{\beta p}n_{\beta l}(\partial_ln_{\alpha j}-\partial_jn_{\alpha l})\partial_in_{\alpha p}\}\|_{L^{2}}\Big)
\bigg)\|\Delta^{s}\nabla\vv\|_{L^{2}}.
\end{align}
Reusing Lemma \ref{pde} and the Sobolev embedding inequality, it follows that
\begin{align*}
   \|\Delta^{s}(\partial_jn_{\alpha k}\partial_in_{\alpha k})\|_{L^{2}}
   \leq&C\|\Delta^{s}(\partial_jn_{\alpha k})\|_{L^{2}}\|\partial_in_{\alpha k}\|_{L^{\infty}}\\
   &+\|\partial_jn_{\alpha k}\|_{L^{\infty}}\|\Delta^{s}(\partial_in_{\alpha k})\|_{L^{2}}\\
   \leq& C\|\nabla\Fp\|_{L^{\infty}}\|\Delta^{s}\nabla\Fp\|_{L^{2}}
   \leq CE^{\frac{1}{2}}_{s}(\Fp,\vv)D^{\frac{1}{2}}_{s}(\Fp,\vv),\\
    \|\Delta^{s}((\nabla\cdot\nn_{\alpha})\partial_in_{\alpha j})\|_{L^{2}}\leq& C\|\nabla\Fp\|_{L^{\infty}}\|\Delta^{s}\nabla\Fp\|_{L^{2}}\le CE^{\frac{1}{2}}_{s}(\Fp,\vv)D^{\frac{1}{2}}_{s}(\Fp,\vv),\\
    \|\Delta^{s}((\partial_jn_{\alpha p}-\partial_pn_{\alpha j})\partial_in_{\alpha p})\|_{L^{2}}\leq& C\|\nabla\Fp\|_{L^{\infty}}\|\Delta^{s}\nabla\Fp\|_{L^{2}}\leq CE^{\frac{1}{2}}_{s}(\Fp,\vv)D^{\frac{1}{2}}_{s}(\Fp,\vv).
\end{align*}
By a similar argument, we obtain the following estimate:
\begin{align*}
    &\|\Delta^{s}\{n_{\beta j}n_{\beta l}(\partial_pn_{\alpha l}-\partial_ln_{\alpha p})\partial_in_{\alpha p}\}\|_{L^{2}}\nonumber\\
    &\quad\leq\|n_{\beta l}(\partial_pn_{\alpha l}-\partial_ln_{\alpha p})\partial_in_{\alpha p}\|_{L^{\infty}}\|\Delta^{s}n_{\beta j}\|_{L^{2}}\nonumber\\
    &\qquad+\|n_{\beta j}(\partial_pn_{\alpha l}-\partial_ln_{\alpha p})\partial_in_{\alpha p}\|_{L^{\infty}}\|\Delta^{s}n_{\beta l}\|_{L^{2}}\nonumber\\
    &\qquad+\|n_{\beta j}n_{\beta l}\partial_in_{\alpha p}\|_{L^{\infty}}\|\Delta
    ^{s}(\partial_pn_{\alpha l}-\partial_ln_{\alpha p})\|_{L^{2}}\nonumber\\
    &\qquad+\|n_{\beta j}n_{\beta l}(\partial_pn_{\alpha l}-\partial_ln_{\alpha p})\|_{L^{\infty}}\|\Delta^{s}(\partial_in_{\alpha p})\|_{L^{2}}\nonumber\\
    &\quad\leq C(\|\nabla\Fp\|^{2}_{L^{\infty}}\|\Delta^{s}\Fp\|_{L^{2}}+\|\nabla\Fp\|_{L^{\infty}}\|\Delta^{s}\nabla\Fp\|_{L^{2}})\nonumber\\
    &\quad\leq C(E^{\frac{1}{2}}_{s}(\Fp,\vv)+E_{s}(\Fp,\vv))D^{\frac{1}{2}}_{s}(\Fp,\vv).
\end{align*}
Further, we also have
\begin{align*}
    \|\Delta^{s}\{n_{\beta p}n_{\beta l}(\partial_ln_{\alpha j}-\partial_jn_{\alpha l})\partial_in_{\alpha p}\}\|_{L^{2}}\leq C\Big(E^{\frac{1}{2}}_{s}(\Fp,\vv)+E_{s}(\Fp,\vv)\Big)D^{\frac{1}{2}}_{s}(\Fp,\vv).
\end{align*}
Then, substituting the above estimates into (\ref{III-expression}) yields 
\begin{align}\label{III-f}
    \Rmnum{3}\leq C\Big(E^{\frac{1}{2}}_{s}(\Fp,\vv)+E_{s}(\Fp,\vv)\Big)D_{s}(\Fp,\vv).
\end{align}
Therefore, combining (\ref{dv}) with (\ref{I-estimate}), (\ref{II-expression-f}) and (\ref{III-f}), we immediately obtain 
\begin{align}\label{dv.}
&\frac{1}{2}\frac{\ud}{\ud t}\|\Delta^{s}\vv\|_{L^{2}}^{2}+\eta\|\nabla\Delta^{s}\vv\|_{L^{2}}^{2}\nonumber\\
&\quad\leq-\Big\langle\frac{\eta_{3}}{\chi_{3}}(\CH^{\Delta^{s}}_{3})\sss_{3}+\frac{\eta_{2}}{\chi_{2}}(\CH^{\Delta^{s}}_{2})\sss_{4}+\frac{\eta_{1}}{\chi_{1}}(\CH^{\Delta^{s}}_{1})\sss_{5},\Delta^{s}\DD\Big\rangle\nonumber\\
&\qquad-\frac{1}{2}\Big\langle(\CH^{\Delta^{s}}_{3})\aaa_{1}+(\CH^{\Delta^{s}}_{2})\aaa_{2}+(\CH^{\Delta^{s}}_{1})\aaa_{3},\Delta^{s}\WW\Big\rangle\nonumber\\
&\qquad+C\sum^3_{k=1}E^{\frac{k}{2}}_{s}(\Fp,\vv)D_{s}(\Fp,\vv).
\end{align}

\subsection{Estimates of higher order terms for the frame $\Fp$}

We now deal with the estimate of higher-order derivatives for the frame $\Fp=(\nn_{1},\nn_{2},\nn_{3})$. First of all, using the equation $(\ref{n1.})$ and integrating by parts, one has 
\begin{align}\label{314}
&\frac{1}{2}\frac{\ud}{\ud t}\langle\nabla\Delta^{s}\nn_{1},\nabla\Delta^{s}\nn_{1}\rangle\nonumber\\
&\quad=-\Big\langle\Delta^{s}\Big[\Big(\frac{1}{2}\WW\cdot\aaa_{1}+\frac{\eta_{3}}{\chi_{3}}\DD\cdot\sss_{3}\Big)\nn_{2}\Big],\Delta^{s+1}\nn_{1}\Big\rangle+\frac{1}{\chi_{3}}\big\langle\Delta^{s}\big((\ML_3\CF_{Bi})\nn_{2}\big),\Delta^{s+1}\nn_{1}\big\rangle\nonumber\\
&\qquad+\Big\langle\Delta^{s}\Big[(\frac{1}{2}\WW\cdot\aaa_{2}+\frac{\eta_{2}}{\chi_{2}}\DD\cdot\sss_{4})\nn_{3}\Big],\Delta^{s+1}\nn_{1}\Big\rangle-\frac{1}{\chi_{2}}\big\langle\Delta^{s}\big((\ML_2\CF_{Bi})\nn_{3}\big),\Delta^{s+1}\nn_{1}\big\rangle\nonumber\\
&\qquad-\langle\nabla\Delta^{s}(\vv\cdot\nabla\nn_{1}),\nabla\Delta^{s}\nn_{1}\rangle\nonumber\\
&\quad\eqdefa\Rmnum{1}_{1}+\Rmnum{1}_{2}+\Rmnum{1}_{3}+\Rmnum{1}_{4}+\Rmnum{1}_{5}.
\end{align}
In the same way, we immediately obtain
\begin{align}
&\frac{1}{2}\frac{\ud}{\ud t}\langle\Delta^{s}\text{div}\nn_{1},\Delta^{s}\text{div}\nn_{1}\rangle\nonumber\\
&\quad=-\Big\langle\Delta^{s}\Big[\Big(\frac{1}{2}\WW\cdot\aaa_{1}+\frac{\eta_{3}}{\chi_{3}}\DD\cdot\sss_{3}\Big)\nn_{2}\Big],\Delta^{s}\nabla\text{div}\nn_{1}\Big\rangle+\frac{1}{\chi_{3}}\langle\Delta^{s}\big((\ML_3\CF_{Bi})\nn_{2}\big),\Delta^{s}\nabla\text{div}\nn_{1}\rangle\nonumber\\
&\qquad+\Big\langle\Delta^{s}\Big[\Big(\frac{1}{2}\WW\cdot\aaa_{2}+\frac{\eta_{2}}{\chi_{2}}\DD\cdot\sss_{4}\Big)\nn_{3}\Big],\Delta^{s}\nabla\text{div}\nn_{1}\Big\rangle-\frac{1}{\chi_{2}}\langle\Delta^{s}\big((\ML_2\CF_{Bi})\nn_{3}\big),\Delta^{s}\nabla\text{div}\nn_{1}\rangle\nonumber\\
&\qquad-\langle\Delta^s\text{div}(\vv\cdot\nabla\nn_1),\Delta^s\text{div}\nn_1\rangle\nonumber\\
&\quad\eqdefa\widetilde{\Rmnum{1}}_{1}+\widetilde{\Rmnum{1}}_{2}+\widetilde{\Rmnum{1}}_{3}+\widetilde{\Rmnum{1}}_{4}+\widetilde{\Rmnum{1}}_{5}.
\end{align}

Using $\nabla\cdot\vv=0$ and Lemma \ref{pde}, we deduce that 
\begin{align}\label{316}
\Rmnum{1}_{5}+\widetilde{\Rmnum{1}}_{5}=&-\langle[\nabla\Delta^{s},\vv\cdot]\nabla\nn_{1},\nabla\Delta^{s}\nn_{1}\rangle-\langle[\Delta^{s}\text{div},\vv\cdot]\nabla\nn_{1},\Delta^{s}\text{div}\nn_{1}\rangle\nonumber\\
\leq&\|[\nabla\Delta^{s},\vv\cdot]\nabla\nn_{1}\|_{L^{2}}\|\nabla\Delta^{s}\nn_{1}\|_{L^{2}}+\|[\Delta^{s}\text{div},\vv\cdot]\nabla\nn_{1}\|_{L^{2}}\|\Delta^{s}\text{div}\nn_{1}\|_{L^{2}}\nonumber\\
\leq&C\big(\|\Delta^{s}\nabla\vv\|_{L^{2}}\|\nabla\Fp\|_{L^{\infty}}+\|\nabla\vv\|_{L^{\infty}}\|\Delta^{s}\nabla\Fp\|_{L^{2}}\big)\|\Delta^{s}\nabla\Fp\|_{L^{2}}\nonumber\\
\leq&CE^{\frac{1}{2}}_{s}(\Fp,\vv)D_{s}(\Fp,\vv).
\end{align}
From the equation (\ref{n1.}) and the expressions of $\hh_i$, one has
\begin{align*}
\|\partial_{t}\nn_{1}\|_{L^{\infty}}
\leq&C\Big(\|\nabla\vv\|_{L^{\infty}}\|\Fp\|_{L^{\infty}}^{3}+\|(\ML_3\CF_{B{i}})\nn_{2}\|_{L^{\infty}}\nonumber\\
&+\|(\ML_2\CF_{B{i}})\nn_{3}\|_{L^{\infty}}+\|\vv\|_{L^{\infty}}\|\nabla\nn_{1}\|_{L^{\infty}}\Big)\nonumber\\
\leq&C\Big(\|\nabla\vv\|_{L^{\infty}}+\|\vv\|_{L^{\infty}}\|\nabla\Fp\|_{L^{\infty}}+\|\nabla\Fp\|_{L^{\infty}}^{2}+\|\Delta\Fp\|_{L^{\infty}}\Big)\nonumber\\
\leq&C\big(E^{\frac{1}{2}}_{s}(\Fp,\vv)+E_{s}(\Fp,\vv)\big).
\end{align*}
Similar to the above derivation, we conclude that
\begin{align}\label{partial n1}
\|\partial_{t}\nn_{i}\|_{L^{\infty}}\leq&C\big(E^{\frac{1}{2}}_{s}(\Fp,\vv)+E_{s}(\Fp,\vv)\big),\quad i=1,2,3. 
\end{align}
By a direct calculation, for $j=1,2,3$, we infer from (\ref{partial n1}) that 
\begin{align}\label{nn-j-Delta-snn1}
&\frac{1}{2}\frac{\ud}{\ud t}\|\nn_{j}\cdot\Delta^{s}(\nabla\times\nn_{1})\|_{L^{2}}^{2}\nonumber\\
&\quad=\Big\langle\partial_{t}\nn_{j}\cdot\Delta^{s}(\nabla\times\nn_{1})+\nn_{j}\cdot(\Delta^{s}\nabla\times\dot{\nn}_{1}), \nn_{j}\cdot\Delta^{s}(\nabla\times\nn_{1})\Big\rangle\nonumber\\
&\qquad-\big\langle\nn_{j}\cdot(\Delta^{s}\nabla\times(\vv\cdot\nabla\nn_{1})),\nn_{j}\cdot\Delta^{s}(\nabla\times\nn_{1})\big\rangle\nonumber\\
&\quad\leq\|\partial_{t}\nn_{j}\|_{L^{\infty}}\|\nn_{j}\|_{L^{\infty}}\|\Delta^{s}\nabla\nn_{1}\|_{L^{2}}^{2}\nonumber\\
&\qquad+\big\langle\nn_{j}\cdot(\Delta^{s}\nabla\times\dot{\nn}_{1}),\nn_{j}\cdot\Delta^{s}(\nabla\times\nn_{1})\big\rangle\nonumber\\
&\qquad-\Big\langle\nn_{j}\cdot(\Delta^{s}\nabla\times(\vv\cdot\nabla\nn_{1})),\nn_{j}\cdot\Delta^{s}(\nabla\times\nn_{1})\Big\rangle\nonumber\\
&\quad\leq C\big(E^{\frac{1}{2}}_s(\Fp,\vv)+E_s(\Fp,\vv)\big)D_s(\Fp,\vv)\nonumber\\
&\qquad+\langle\nn_{j}\cdot(\Delta^{s}\nabla\times\dot{\nn}_{1}),\nn_{j}\cdot\Delta^{s}(\nabla\times\nn_{1})\rangle\nonumber\\
&\qquad-\Big\langle\nn_{j}\cdot(\Delta^{s}\nabla\times(\vv\cdot\nabla\nn_{1})),\nn_{j}\cdot\Delta^{s}(\nabla\times\nn_{1})\Big\rangle,
\end{align}
where $\dot{\nn}_1=(\partial_t+\vv\cdot\nabla)\nn_1$. The last term in (\ref{nn-j-Delta-snn1}) can be estimated as 
\begin{align*}
&-\langle\nn_{j}\cdot(\Delta^{s}\nabla\times(\vv\cdot\nabla\nn_{1})),\nn_{j}\cdot\Delta^{s}(\nabla\times\nn_{1})\rangle\\
&\quad\leq\|\nn_{j}\cdot(\Delta^{s}\nabla\times(\vv\cdot\nabla\nn_{1}))\|_{
L^{2}}\|\nn_{j}\cdot\Delta^{s}(\nabla\times\nn_{1})\|_{L^{2}}\\
&\quad\leq C\|(\Delta^{s}\nabla\times(\vv\cdot\nabla\nn_{1}))\|_{L^{2}}\|\Delta^{s}\nabla\nn_{1}\|_{L^{2}}\\
&\quad\leq C\Big(\|\nabla\nn_{1}\|_{L^{\infty}}\|\Delta^{s}\nabla\vv\|_{L^{2}}+\|\vv\|_{L^{\infty}}\|\Delta^{s+1}\nn_{1}\|_{L^{2}}
\Big)\|\Delta^s\nabla\nn_1\|_{L^2}\\
&\quad\leq CE^{\frac{1}{2}}_s(\Fp,\vv)D_s(\Fp,\vv).
\end{align*}
Consequently, using (\ref{nn-j-Delta-snn1}) and the above inequality, one can get that
\begin{align}\label{J}
\frac{1}{2}\frac{\ud}{\ud t}\|\nn_{j}\cdot\Delta^{s}(\nabla\times\nn_{1})\|_{L^{2}}^{2}
\leq &C\big(E^{\frac{1}{2}}_{s}(\Fp,\vv)+E_{s}(\Fp,\vv)\big)D_{s}(\Fp,\vv)\nonumber\\
&+\underset{J}{\underbrace{\big\langle\nn_{j}\cdot(\Delta^{s}\nabla\times\dot{\nn}_{1}),\nn_{j}\cdot\Delta^{s}(\nabla\times\nn_{1})\big\rangle}}.
\end{align}

We are now ready to handle the term $J$ in (\ref{J}). By means of the equation (\ref{n1.}), the term $J$ can be further expressed as 
\begin{align}\label{J-expression}
J=&\big\langle\Delta^{s}\nabla\times\dot{\nn}_{1},(\nn_{j}\otimes\nn_j)\cdot\Delta^{s}(\nabla\times\nn_{1})\big\rangle\nonumber\\
=&\big\langle\Delta^{s}\nabla\times\dot{\nn}_{1},\Delta^{s}(\nabla\times\nn_{1}\cdot(\nn_{j}\otimes\nn_j))\big\rangle-\big\langle\Delta^{s}\nabla\times\dot{\nn}_{1},[\Delta^{s},(\nn_{j}\otimes\nn_j)](\nabla\times\nn_{1})\big\rangle\nonumber\\
=&\big\langle\Delta^{s}\dot{\nn}_{1},\Delta^{s}\nabla\times(\nabla\times\nn_{1}\cdot(\nn_{j}\otimes\nn_j))\big\rangle-\big\langle\Delta^{s}\nabla\times\dot{\nn}_{1},[\Delta^{s},(\nn_{j}\otimes\nn_j)](\nabla\times\nn_{1})\big\rangle\nonumber\\
=&\underset{J_{1}}{\underbrace{\Big\langle\Delta^{s}\Big[\Big(\frac{1}{2}\WW\cdot\aaa_{1}+\frac{\eta_{3}}{\chi_{3}}\DD\cdot\sss_{3}\Big)\nn_{2}\Big],\Delta^{s}\nabla\times(\nabla\times\nn_{1}\cdot(\nn_{j}\otimes\nn_j))\Big\rangle} }\nonumber\\
&\underset{J_{2}}{\underbrace{-\frac{1}{\chi_{3}}\Big\langle\Delta^{s}(
(\ML_3\CF_{B_{i}})\nn_{2}),\Delta^{s}\nabla\times(\nabla\times\nn_{1}\cdot(\nn_{j}\otimes\nn_j))\Big\rangle}}\nonumber\\
&\underset{J_{3}}{\underbrace{-\Big\langle\Delta^{s}\Big[\Big(\frac{1}{2}\WW\cdot\aaa_{2}+\frac{\eta_{2}}{\chi_{2}}\DD\cdot\sss_{4}\Big)\nn_{3}\Big],\Delta^{s}\nabla\times(\nabla\times\nn_{1}\cdot(\nn_{j}\otimes\nn_j))\Big\rangle} }\nonumber\\
&\underset{J_{4}}{\underbrace{+\frac{1}{\chi_{2}}\Big\langle\Delta^{s}(
(\ML_2\CF_{B_{i}})\nn_{3}),\Delta^{s}\nabla\times(\nabla\times\nn_{1}\cdot(\nn_{j}\otimes\nn_j))\Big\rangle}}\nonumber\\
&\underset{J_{5}}{\underbrace{-\big\langle\Delta^{s}\dot{\nn}_{1},\nabla\times([\Delta^{s},(\nn_{j}\otimes\nn_j)](\nabla\times\nn_{1}))\big\rangle}}.
\end{align}
Now we estimate (\ref{J-expression}) term by term as follows. To begin with, we handle the term $J_5$. For this purpose, for $i\neq j$, armed with Lemma \ref{pde} we estimate the following terms related to $\nabla\vv$:
\begin{align*}
    \|\Delta^{s}(\WW\cdot\aaa_{j})\|_{L^{2}}&\leq C(\|[\Delta^{s},\aaa_{j}\cdot]\WW\|_{L^{2}}+\|\aaa_{j}\cdot\Delta^{s}\WW\|_{L^{2}})\\
    &\leq C(\|\nabla^{2s}\aaa_{j}\|_{L^{2}}\|\WW\|_{L^{\infty}}+\|\nabla\aaa_{j}\|_{L^{\infty}}\|\nabla^{2s-1}\WW\|_{L^{2}}+\|\Delta^{s}\nabla\vv\|_{L^{2}})\\
    &\leq C(\|\nabla^{2s}\Fp\|_{L^{2}}\|\nabla\vv\|_{H^{2}}+\|\nabla\Fp\|_{H^{2}}\|\nabla^{2s}\vv\|_{L^{2}}+\|\Delta^{s}\nabla\vv\|_{L^{2}})\\
    &\leq \big(1+E^{\frac{1}{2}}_{s}(\Fp,\vv)\big)D^{\frac{1}{2}}_{s}(\Fp,\vv),\\
    \|\Delta^{s}(\DD\cdot\sss_{j})\|_{L^{2}}&\leq C(\|[\Delta^{s},\sss_{j}\cdot]\DD\|_{L^{2}}+\|\sss_{j}\cdot\Delta^{s}\DD\|_{L^{2}})\\
    &\leq C(\|\nabla^{2s}\Fp\|_{L^{2}}\|\nabla\vv\|_{H^{2}}+\|\nabla\Fp\|_{H^{2}}\|\nabla^{2s}\vv\|_{L^{2}}+\|\Delta^{s}\nabla\vv\|_{L^{2}})\\
    &\leq C\big(1+E^{\frac{1}{2}}_{s}(\Fp,\vv)\big)D^{\frac{1}{2}}_{s}(\Fp,\vv),\\
    \|[\Delta^{s},\nn_{i}]\WW\cdot\aaa_{j}\|_{L^{2}}
    &\leq C\Big(\|\nabla^{2s}\nn_{i}\|_{L^{2}}\|\WW\cdot\aaa_{j}\|_{L^{\infty}}+\|\nabla\nn_{i}\|_{L^{\infty}}\|\nabla^{2s-1}(\WW\cdot\aaa_{j})\|_{L^{2}}\Big)\\
    &\leq C\Big(\|\nabla^{2s}\Fp\|_{L^{2}}\|\nabla\vv\|_{H^{2}}+\|\nabla\Fp\|_{H^{2}}(\|\nabla\vv\|_{L^{\infty}}\|\nabla^{2s-1}\Fp\|_{L^{2}}+\|\nabla^{2s}\vv\|_{L^{2}})\Big)\\
    &\leq C\big(E^{\frac{1}{2}}_{s}(\Fp,\vv)+E_{s}(\Fp,\vv)\big)D^{\frac{1}{2}}_{s}(\Fp,\vv),\\
 \|[\Delta^{s},\nn_{i}]\DD\cdot\aaa_{j}\|_{L^{2}}
    &\leq C\Big(\|\nabla^{2s}\nn_{i}\|_{L^{2}}\|\DD\cdot\aaa_{j}\|_{L^{\infty}}+\|\nabla\nn_{i}\|_{L^{\infty}}\|\nabla^{2s-1}(\DD\cdot\aaa_{j})\|_{L^{2}}\Big)\\
    &\leq C\big(E^{\frac{1}{2}}_{s}(\Fp,\vv)+E_{s}(\Fp,\vv)\big)D^{\frac{1}{2}}_{s}(\Fp,\vv).
\end{align*}
Again, for $j\neq k$, from (\ref{Delta-s-nni-hhj}) and Lemma \ref{pde} we estimate the following terms related to $\ML_k\CF_{Bi}$:
\begin{align*}
    \|[\Delta^{s},\nn_{j}]\ML_{k}\CF_{Bi}\|_{L^{2}}
    \leq& C\Big(\|\nabla^{2s+2}\Fp\|_{L^{2}}\|\ML_{k}\CF_{Bi}\|_{L^{\infty}}+\|\nabla\nn_{j}\|_{L^{\infty}}\|\nabla^{2s+1}(\ML_{k}\CF_{Bi})\|_{L^{2}}\Big)\\
    \leq&C\Big(\|\nabla^{2s+2}\Fp\|_{L^2}\|\hh_j\|_{L^{\infty}}+\|\nabla\nn_j\|_{L^{\infty}}\|\nabla^{2s+1}(\nn_i\cdot\hh_j)\|_{L^2}\Big)\\
    \leq&C\Big(\|\nabla^{2s+2}\Fp\|_{L^2}\|\hh_j\|_{H^{2}}
    +\|\nabla\Fp\|_{H^{2}}(\|[\nabla^{2s+1},\nn_{i}\cdot]\hh_{j}\|_{L^{2}}+\|\nn_{i}\cdot\nabla^{2s+1}\hh_{j}\|_{L^{2}})\Big)\\
    \leq&C\Big(\|\nabla^{2s+2}\Fp\|_{L^2}(\|\Delta\Fp\|_{L^{2}}+\|\Delta^{s}\Fp\|_{L^{2}}+\|\nabla\Fp\|^{2}_{L^{2}})\\
    &+\|\nabla\Fp\|_{H^{2}}(\|\nabla^{2s+1}\nn_{i}\|_{L^{2}}\|\hh_{j}\|_{L^{\infty}}+\|\nabla\nn_{i}\|_{L^{\infty}}\|\nabla^{2s+1}\hh_{j}\|_{L^{2}}+\|\nabla^{2s+1}\hh_{j}\|_{L^{2}})\Big)\\
    \leq&C\Big(1+\sum^3_{k=1}E^{\frac{k}{2}}_{s}(\Fp,\vv)\Big)D^{\frac{1}{2}}_{s}(\Fp,\vv),\\
    \|\Delta^{s}(\ML_k\CF_{B{i}})\|_{L^{2}}=&\Big\|\nn_{i}\cdot\Delta^{s}\hh_{j}-\nn_{j}\cdot\Delta^{s}\hh_{i}
+[\Delta^{s},\nn_{i}\cdot]\hh_{j}+[\Delta^{s},\nn_{j}\cdot]\hh_{i}\Big\|_{L^{2}}\\
\leq&C\big(\|\CH_{3}^{\Delta^{s}}\|_{L^{2}}+\|[\Delta^{s},\nn_{i}\cdot]\hh_{j}\|_{L^{2}}+\|[\Delta^{s},\nn_{j}\cdot]\hh_{i}\|_{L^{2}}\big)\\
\leq&C\big(\|\Delta^{s+1}\Fp\|_{L^{2}}+\|\nabla\Fp\|_{L^{\infty}}\|\Delta^{s}\nabla\Fp\|_{L^{2}}+\|\nabla\Fp\|_{L^{\infty}}^{2}\|\Delta^{s}\Fp\|_{L^{2}}\big)\\
&+C\big(\|\nabla^{2s+1}\nn_{i}\|_{L^{2}}\|\hh_{j}\|_{L^{\infty}}+\|\nabla\nn_{i}\|_{L^{\infty}}\|\nabla^{2s+1}\hh_{j}\|_{L^{2}}+\|\nabla^{2s+1}\hh_{j}\|_{L^{2}}\big)\\
\leq&C\Big(1+\sum^3_{k=1}E^{\frac{k}{2}}_{s}(\Fp,\vv)\Big)D^{\frac{1}{2}}_{s}(\Fp,\vv),
\end{align*}
where $\CH^{\Delta^s}_3$ is expressed by (\ref{CH}).

Therefore, from the equation (\ref{n1.}) and  Lemma \ref{pde} together with the above estimates, the term $\|\Delta^s\dot{\nn}_1\|_{L^2}$  can be handled as
\begin{align}\label{Delta-dot-nn1}
    \|\Delta^{s}\dot{\nn}_{1}\|_{L^{2}}\leq&\Big\|\Delta^s\Big\{\Big(\frac{1}{2}\WW\cdot\aaa_{1}+\frac{\eta_{3}}{\chi_{3}}\DD\cdot\sss_{3}-\frac{1}{\chi_{3}}\ML_3\CF_{Bi}\Big)\nn_{2}\Big\}\Big\|_{L^{2}}\nonumber\\
 &+\Big\|\Delta^s\Big\{\Big(\frac{1}{2}\WW\cdot\aaa_{2}+\frac{\eta_{2}}{\chi_{2}}\DD\cdot\sss_{4}-\frac{1}{\chi_{2}}\ML_2\CF_{Bi}\Big)\nn_{3}\Big\}\Big\|_{L^{2}}\nonumber\\
\leq&C\Big(\sum_{j=1,2}\|\Delta^s(\WW\cdot\aaa_{j})\|_{L^{2}}+\sum_{j=3,4}\|\Delta^{s}(\DD\cdot\sss_{j})\|_{L^{2}}\nonumber\\
 &+\|[\Delta^{s},\nn_{2}](\WW\cdot\aaa_{1})\|_{L^{2}}+\|[\Delta^s,\nn_{3}](\WW\cdot\aaa_{2})\|_{L^{2}}\nonumber\\
 &+\|[\Delta^{s},\nn_{2}](\DD\cdot\sss_{3})\|_{L^{2}}+\|[\Delta^{s},\nn_{3}](\DD\cdot\sss_{4})\|_{L^{2}}\nonumber\\
 &+\|[\Delta^{s},\nn_{2}]\ML_{3}\CF_{Bi}\|_{L^{2}}+\|[\Delta^{s},\nn_{3}]\ML_{2}\CF_{Bi}\|_{L^{2}}\nonumber\\
 &+\sum_{k=2,3}\|\Delta^{s}(\ML_{k}\CF_{Bi})\|_{L^{2}}\Big)\nonumber\\
\leq&C\Big(1+\sum^3_{k=1}E^{\frac{k}{2}}_{s}(\Fp,\vv)\Big)D^{\frac{1}{2}}_{s}(\Fp,\vv),
\end{align}
Further, the term $J_5$ can be estimated as
\begin{align*}
J_5
\leq&\|\Delta^{s}\dot{\nn}_{1}\|_{L^{2}}\|\nabla\times([\Delta^{s},(\nn_{j}\otimes\nn_j)](\nabla\times\nn_{1}))\|_{L^{2}}\\
\leq&C\|\Delta^{s}\dot{\nn}_{1}\|_{L^{2}}\sum_{i=1}^{3}\Big(\|[\Delta^{s},\partial_{i}(\nn_{j}\otimes\nn_j)](\nabla\times\nn_{1}))\|_{L^{2}}+\|[\Delta^{s},\nn_{j}\otimes\nn_j](\nabla\times(\partial_{i}\nn_{1}))\|_{L^{2}}\Big)\\
\leq &C\sum_{k=1}^{5}E^{\frac{k}{2}}_{s}(\Fp,\vv)D_{s}(\Fp,\vv).
\end{align*}
Consequently, from (\ref{J}) and (\ref{J-expression}) together with the estimate of $J_5$ one can obtain
\begin{align}\label{320}
&\frac{1}{2}\frac{\ud}{\ud t}\|\nn_{j}\cdot\Delta^{s}(\nabla\times\nn_{1})\|_{L^{2}}^{2}\nonumber\\
&\quad\leq J_{1}+J_{2}+J_{3}+J_{4}+C\sum_{k=1}^{5}E^{\frac{k}{2}}_{s}(\Fp,\vv)D_{s}(\Fp,\vv).
\end{align}

For simplicity, for $i=1,2,3$ we define 
\begin{align*}
    \Delta^s\hh_i\eqdefa\CL_i(\Fp)+\CG_i(\Fp),
\end{align*}
where $\CL_i(\Fp)$ and $\CG_i(\Fp)$ are higher-order and lower-order derivative terms, respectively,
\begin{align}
\CL_i(\Fp)=&\gamma_i\Delta^{s+1}\nn_i+k_i\Delta^s\nabla\text{div}\nn_i-\sum^3_{j=1}k_{ji}\Delta^s\nabla\times(\nabla\times\nn_i\cdot(\nn_j\otimes\nn_j)),\nonumber\\
\mathcal{G}_{i}(\Fp)=&-\sum_{j=1}^{3}k_{ij}\Delta^{s}\big[(\nn_{i}\cdot\nabla\times\nn_{j})(\nabla\times\nn_{j})\big].  \label{mL}
\end{align}
Then, using the definition of $\hh_{1}$ and $(\ref{mL})$, we infer that
\begin{align}\label{beta-4-Jbeta}
&\sum_{\beta=1}^{4}\Big(\gamma_{1}\Rmnum{1}_{\beta}+k_{1}\widetilde{\Rmnum{1}}_{\beta}+\sum_{j=1}^{3}k_
{j1}J_{\beta}\Big)=-\langle\Delta^{s}\dot{\nn}_{1},\Delta^s\hh_1-\mathcal{G}_{1}(\Fp)\rangle\nonumber\\
&\quad\leq-\Big\langle\Big(\frac{1}{2}\Delta^{s}\WW\cdot\aaa_{1}+\frac{\eta_{3}}{\chi_{3}}\Delta^{s}\DD\cdot\sss_{3}\Big)\nn_{2},\Delta^{s}\hh_{1}\Big\rangle\nonumber\\
&\qquad+\Big\langle\Big(\frac{1}{2}\Delta^{s}\WW\cdot\aaa_{2}+\frac{\eta_{2}}{\chi_{2}}\Delta^{s}\DD\cdot\sss_{4}\Big)\nn_{3},\Delta^{s}\hh_{1}\Big\rangle\nonumber\\
&\qquad+\Big\langle-\frac{1}{\chi_{2}}(\CH^{\Delta^{s}}_{2})\nn_{3}+\frac{1}{\chi_{3}}(\CH^{\Delta^{s}}_{3})\nn_{2},\Delta^{s}\hh_{1}\Big\rangle\nonumber\\
&\qquad+C\underset{\mathcal{A}_{1}}{\underbrace{\Big|\langle\nabla_{l}[\Delta^{s},\aaa_{1}\otimes\nn_{2}\cdot]\WW+\nabla_{l}[\Delta^{s},\aaa_{2}\otimes\nn_{3}\cdot]\WW,\Delta^{s-1}\nabla^{l}\hh_{1}\rangle\Big|}}\nonumber\\
&\qquad+C\underset{\mathcal{A}_{2}}{\underbrace{\Big|\langle\nabla_{l}[\Delta^{s},\sss_{3}\otimes\nn_{2}\cdot]\DD+\nabla_{l}[\Delta^{s},\sss_{4}\otimes\nn_{3}\cdot]\WW,\Delta^{s-1}\nabla^{l}\hh_{1}\rangle\Big|}}\nonumber\\
&\qquad+C\underset{\mathcal{A}_{3}}{\underbrace{\Big|\langle\nabla_{l}[\Delta^{s},\nn_{2}\cdot](\ML_3\CF_{B_{i}})-\nabla_{l}[\Delta^{s},\nn_{3}\cdot](\ML_2\CF_{B_{i}}),\Delta^{s-1}\nabla^{l}\hh_{1}\rangle\Big|}}\nonumber\\
&\qquad+\underset{\mathcal{B}_{1}}{\underbrace{\frac{1}{\chi_{3}}\Big\langle([\Delta^{s},\nn_{1}\cdot]\hh_{2}-[\Delta^{s},\nn_{2}\cdot]\hh_{1})\nn_{2},\Delta^{s}\hh_{1}\Big\rangle}}\nonumber\\
&\qquad-\underset{\mathcal{B}_{2}}{\underbrace{\frac{1}{\chi_{2}}\Big\langle([\Delta^{s},\nn_{3}\cdot]\hh_{1}-[\Delta^{s},\nn_{1}]\hh_{3})\nn_{3},\Delta^{s}\hh_{1}\Big\rangle}}\nonumber\\
&\qquad+\|\Delta^{s}\dot{\nn}_{1}\|_{L^{2}}\|\mathcal{G}_{1}(\Fp)\|_{L^{2}}.
\end{align}
For the term $\mathcal{A}_{1}$, we have
\begin{align*}
\mathcal{A}_{1}\leq&\Big(\|\langle\nabla_{l}[\Delta^{s},\aaa_{1}\otimes\nn_{2}\cdot]\WW\|_{L^{2}}+\|\nabla_{l}[\Delta^{s},\aaa_{2}\otimes\nn_{3}\cdot]\WW\|_{L^{2}}\Big)\|\Delta^{s-1}\nabla^{l}\hh_{1}\|_{L^{2}}\\
\leq&\Big(\|[\Delta^{s},\nabla_{l}(\aaa_{1}\otimes\nn_{2})\cdot]\WW\|_{L^{2}}+\|[\Delta^{s},\aaa_{1}\otimes\nn_{2}\cdot]\nabla_{l}\WW\|_{L^{2}}\\
&+\|[\Delta^{s},\nabla_{l}(\aaa_{2}\otimes\nn_{3})\cdot]\WW\|_{L^{2}}+\|[\Delta^{s},\aaa_{2}\otimes\nn_{3}\cdot]\nabla_{l}\WW\|_{L^{2}}\Big)\|\Delta^{s-1}\nabla^{l}\hh_{1}\|_{L^{2}}\\
\leq&C\Big(\|\Delta^{s}\nabla(\aaa_{1}\otimes\nn_{2})\|_{L^{2}}\|\nabla\vv\|_{L^{\infty}}+\|\nabla^{2}(\aaa_{1}\otimes\nn_{2})\|_{L^{\infty}}\|\Delta^{s}\vv\|_{L^{2}}\\
&+\|\Delta^{s}\nabla(\aaa_{2}\otimes\nn_{3})\|_{L^{2}}\|\nabla\vv\|_{L^{\infty}}+\|\nabla^{2}(\aaa_{2}\otimes\nn_{3})\|_{L^{\infty}}\|\Delta^{s}\vv\|_{L^{2}}\\
&+\|\Delta^{s}\nabla(\aaa_{1}\otimes\nn_{2})\|_{L^{2}}\|\nabla\vv\|_{L^{\infty}}+\|\nabla(\aaa_{1}\otimes\nn_{2})\|_{L^{\infty}}\|\Delta^{s}\nabla\vv\|_{L^{2}}\\
&+\|\Delta^{s}\nabla(\aaa_{2}\otimes\nn_{3})\|_{L^{2}}\|\nabla\vv\|_{L^{\infty}}+\|\nabla(\aaa_{2}\otimes\nn_{3})\|_{L^{\infty}}\|\Delta^{s}\nabla\vv\|_{L^{2}}\Big)\|\Delta^{s-1}\nabla^{l}\hh_{1}\|_{L^{2}}\\
\leq&C\sum_{k=1}^{5}E^{\frac{k}{2}}_{s}(\Fp,\vv)D_{s}(\Fp,\vv),
\end{align*}
where we have employed estimates
\begin{align*}
\|\nabla^{2}(\aaa_{1}\otimes\nn_{2})\|_{L^{\infty}}\leq&\|\nabla^{2}\aaa_{1}\|_{
L^{\infty}}+2\|\nabla\aaa_{1}\otimes\nabla\nn_{2}\|_{L^{\infty}}+\|\nabla^{2}\nn_{2}\|_{L^{\infty}}\\
\leq&C\big(\|\nabla^{2}\Fp\|_{
L^{\infty}}+\|\nabla\Fp\|_{
L^{\infty}}^{2}\big)\\
\leq&C\big(E^{\frac{1}{2}}_{s}(\Fp,\vv)+E_{s}(\Fp,\vv)\big),\\
\|\Delta^{s}\nabla(\aaa_{1}\otimes\nn_{2})\|_{L^{2}}\leq&C\big(\|[\Delta^{s}\nabla,\aaa_{1}]\nn_{2}\|_{L^{2}}+\|\aaa_{1}\otimes\Delta^{s}\nabla\nn_{2}\|_{L^{2}}\big)\\
\leq&C\Big(\|\Delta^{s}\nabla\aaa_{1}\|_{L^{2}}\|\nn_{2}\|_{L^{\infty}}+\|\nabla\aaa_{1}\|_{L^{\infty}}\|\Delta^{2}\nn_{2}\|_{L^{2}}+\|\Delta^{s}\nabla\Fp\|_{L^{2}}\Big)\\
\leq&C\big(\|\Delta^{s}\nabla\Fp\|_{L^{2}}+\|\nabla\Fp\|_{L^{\infty}}\|\Delta^{s}\Fp\|_{L^{2}}\big)\\
\leq&C\big(1+E^{\frac{1}{2}}_{s}(\Fp,\vv)\big)D^{\frac{1}{2}}_{s}(\Fp,\vv),\\
\|\Delta^{s-1}\nabla\hh_{1}\|_{L^{2}}\leq&C\big(1+E^{\frac{1}{2}}_{s}(\Fp,\vv)+E_{s}(\Fp,\vv)\big)D^{\frac{1}{2}}_{s}(\Fp,\vv).
\end{align*}
Obviously, the estimate of the term $\CA_2$ is the same as that of the term $\CA_1$.
For the term $\mathcal{A}_{3}$, we infer from Lemma \ref{pde} that
\begin{align*}
    \mathcal{A}_{3}\leq&\Big(\|[\Delta^{s},\nabla_{l}\nn_{2}\cdot]\ML_{3}\CF_{Bi}\|_{L^{2}}+\|[\Delta^{s},\nabla_{l}\nn_{3}\cdot]\ML_{2}\CF_{Bi}\|_{L^{2}}\\
    &+\|[\Delta^{s},\nn_{2}\cdot]\nabla_{l}\ML_{3}\CF_{Bi}\|_{L^{2}}+\|[\Delta^{s},\nn_{3}\cdot]\nabla_{l}\ML_{2}\CF_{Bi}\|_{L^{2}}\Big)\|\Delta^{s-1}\nabla^{l}\hh_{1}\|_{L^{2}}\\
    \leq&C\sum_{k=1}^{5}E^{\frac{k}{2}}_{s}(\Fp,\vv)D_{s}(\Fp,\vv).
\end{align*}
From (\ref{Delta-dot-nn1}) and (\ref{mL}), we immediately obtain
\begin{align*}
    \|\Delta^{s}\dot{\nn}_{1}\|_{L^{2}}\|\mathcal{G}_{1}(\Fp)\|_{L^{2}}\leq C\sum_{k=1}^{5}E^{\frac{k}{2}}_{s}(\Fp,\vv)D_{s}(\Fp,\vv).
\end{align*}
Then, plugging the above estimates $\CA_i(i=1,2,3)$ into (\ref{beta-4-Jbeta}) leads to
\begin{align}\label{321}
&\sum_{\beta=1}^{4}\Big(\gamma_{1}\Rmnum{1}_{\beta}+k_{1}\widetilde{\Rmnum{1}}_{\beta}+\sum_{j=1}^{3}k_
{j1}J_{\beta}\Big)\nonumber\\
&\quad\leq-\Big\langle\Big(\frac{1}{2}\Delta^{s}\WW\cdot\aaa_{1}+\frac{\eta_{3}}{\chi_{3}}\Delta^{s}\DD\cdot\sss_{3}\Big)\nn_{2},\Delta^{s}\hh_{1}\Big\rangle\nonumber\\
&\qquad+\Big\langle\Big(\frac{1}{2}\Delta^{s}\WW\cdot\aaa_{2}+\frac{\eta_{2}}{\chi_{2}}\Delta^{s}\DD\cdot\sss_{4}\Big)\nn_{3},\Delta^{s}\hh_{1}\Big\rangle\nonumber\\
&\qquad+\Big\langle-\frac{1}{\chi_{2}}(\CH^{\Delta^{s}}_{2})\nn_{3}+\frac{1}{\chi_{3}}(\CH^{\Delta^{s}}_{3})\nn_{2},\Delta^{s}\hh_{1}\Big\rangle\nonumber\\
&\qquad+\mathcal{B}_{1}+\mathcal{B}_{2}+C\sum_{k=1}^{5}E^{\frac{k}{2}}_{s}(\Fp,\vv)D_{s}(\Fp,\vv).
\end{align}

Therefore, combining (\ref{314})--(\ref{316}) with (\ref{320}) and (\ref{321}), we arrive at
\begin{align}\label{Es-1-Fp}
\frac{\ud}{\ud t}\CE_{1}^{s}(\Fp)\leq&-\Big\langle\Big(\frac{1}{2}\Delta^{s}\WW\cdot\aaa_{1}+\frac{\eta_{3}}{\chi_{3}}\Delta^{s}\DD\cdot\sss_{3}\Big)\nn_{2},\Delta^{s}\hh_{1}\Big\rangle\nonumber\\
&+\Big\langle\Big(\frac{1}{2}\Delta^{s}\WW\cdot\aaa_{2}+\frac{\eta_{2}}{\chi_{2}}\Delta^{s}\DD\cdot\sss_{4}\Big)\nn_{3},\Delta^{s}\hh_{1}\Big\rangle\nonumber\\
&+\Big\langle-\frac{1}{\chi_{2}}(\CH^{\Delta^{s}}_{2})\nn_{3}+\frac{1}{\chi_{3}}(\CH^{\Delta^{s}}_{3})\nn_{2},\Delta^{s}\hh_{1}\Big\rangle\nonumber\\
&+\mathcal{B}_{1}+\mathcal{B}_{2}+C\sum_{k=1}^{5}E^{\frac{k}{2}}_{s}(\Fp,\vv)D_{s}(\Fp,\vv).    
\end{align}
Similar to the derivation of (\ref{Es-1-Fp}), it follows that
\begin{align}
\frac{\ud}{\ud t}\CE_{2}^{s}(\Fp)\leq&-\Big\langle\Big(\frac{1}{2}\Delta^{s}\WW\cdot\aaa_{1}+\frac{\eta_{3}}{\chi_{3}}\Delta^{s}\DD\cdot\sss_{3}\Big)\nn_{1},\Delta^{s}\hh_{2}\rangle\nonumber\\
&+\Big\langle\Big(\frac{1}{2}\Delta^{s}\WW\cdot\aaa_{3}+\frac{\eta_{1}}{\chi_{1}}\Delta^{s}\DD\cdot\sss_{5}\Big)\nn_{3},\Delta^{s}\hh_{2}\Big\rangle\nonumber\\
&+\Big\langle-\frac{1}{\chi_{3}}(\CH^{\Delta^{s}}_{3})\nn_{1}+\frac{1}{\chi_{1}}(\CH^{\Delta^{s}}_{1})\nn_{3},\Delta^{s}\hh_{2}\Big\rangle\nonumber\\
&+\mathcal{B}'_{1}+\mathcal{B}'_{2}+C\sum_{k=1}^{5}E^{\frac{k}{2}}_{s}(\Fp,\vv)D_{s}(\Fp,\vv),\label{Es-2-Fp}\\
\frac{\ud}{\ud t}\CE_{3}^{s}(\Fp)\leq&-\Big\langle\Big(\frac{1}{2}\Delta^{s}\WW\cdot\aaa_{2}+\frac{\eta_{2}}{\chi_{2}}\Delta^{s}\DD\cdot\sss_{4}\Big)\nn_{1},\Delta^{s}\hh_{3}\Big\rangle\nonumber\\
&+\Big\langle\Big(\frac{1}{2}\Delta^{s}\WW\cdot\aaa_{3}+\frac{\eta_{1}}{\chi_{1}}\Delta^{s}\DD\cdot\sss_{5})\nn_{2},\Delta^{s}\hh_{3}\Big\rangle\nonumber\\
&+\Big\langle-\frac{1}{\chi_{1}}(\CH^{\Delta^{s}}_{1})\nn_{2}+\frac{1}{\chi_{2}}(\CH^{\Delta^{s}}_{2})\nn_{1},\Delta^{s}\hh_{3}\Big\rangle\nonumber\\
&+\mathcal{B}''_{1}+\mathcal{B}''_{2}+C\sum_{k=1}^{5}E^{\frac{k}{2}}_{s}(\Fp,\vv)D_{s}(\Fp,\vv), \label{Es-3-Fp}
\end{align}
where $\mathcal{B}'_{i},\mathcal{B}''_{i}(i=1,2)$ are expressed by, respectively,
\begin{align*}
\mathcal{B}'_{1}=&-\frac{1}{\chi_{3}}\Big\langle([\Delta^{s},\nn_{1}\cdot]\hh_{2}-[\Delta^{s},\nn_{2}]\hh_{1})\nn_{1},\Delta^{s}\hh_{2}\Big\rangle,\\
\mathcal{B}'_{2}=&\frac{1}{\chi_{1}}\Big\langle([\Delta^{s},\nn_{2}\cdot]\hh_{3}-[\Delta^{s},\nn_{3}]\hh_{2})\nn_{3},\Delta^{s}\hh_{2}\Big\rangle,\\
\mathcal{B}''_{1}=&-\frac{1}{\chi_{1}}\Big\langle([\Delta^{s},\nn_{2}\cdot]\hh_{3}-[\Delta^{s},\nn_{3}]\hh_{2})\nn_{2},\Delta^{s}\hh_{3}\Big\rangle,\\
\mathcal{B}''_{2}=&\frac{1}{\chi_{2}}\Big\langle([\Delta^{s},\nn_{3}\cdot]\hh_{1}-[\Delta^{s},\nn_{1}]\hh_{3})\nn_{1},\Delta^{s}\hh_{3}\Big\rangle.
\end{align*}
From the definitions of $\CH^{\Delta^{s}}_{i}(i=1,2,3)$ in (\ref{CH}), direct calculations lead to
\begin{align}\label{CB-estimates}
&\sum_{\alpha=1}^{2}\big(\mathcal{B}_{\alpha}+\mathcal{B}'_{\alpha}+\mathcal{B}''_{\alpha}\big)\nonumber\\
&\quad=-\frac{1}{\chi_{3}}\langle[\Delta^{s},\nn_{1}\cdot]\hh_{2}-[\Delta^{s},\nn_{2}]\hh_{1},\CH^{\Delta^{s}}_{3}\rangle,\nonumber\\
&\qquad-\frac{1}{\chi_{2}}\langle[\Delta^{s},\nn_{3}\cdot]\hh_{1}-[\Delta^{s},\nn_{1}]\hh_{3},\CH^{\Delta^{s}}_{2}\rangle\nonumber\\
&\qquad-\frac{1}{\chi_{1}}\langle[\Delta^{s},\nn_{2}\cdot]\hh_{3}-[\Delta^{s},\nn_{3}]\hh_{2},\Delta^{s}\CH^{\Delta^{s}}_{1}\rangle\nonumber\\
&\quad\leq C\sum_{k=1}^{5}E^{\frac{k}{2}}_{s}(\Fp,\vv)D_{s}(\Fp,\vv).
\end{align}
Consequently, from (\ref{Es-1-Fp})--(\ref{Es-3-Fp}) and (\ref{CB-estimates}) we derive that
\begin{align}\label{Ps}
&\frac{\ud}{\ud t}\sum_{i=1}^{3}\CE_{i}^{s}(\Fp)+\sum_{i=1}^{3}\frac{1}{\chi_{i}}\|\CH^{\Delta^{s}}_{i}\|_{L^{2}}\nonumber\\
&\quad\leq\Big\langle\Big(\frac{1}{2}\Delta^{s}\WW\cdot\aaa_{1}+\frac{\eta_{3}}{\chi_{3}}\Delta^{s}\DD\cdot\sss_{3}\Big),\CH^{\Delta^{s}}_{3}\Big\rangle\nonumber\\
&\qquad+\Big\langle\Big(\frac{1}{2}\Delta^{s}\WW\cdot\aaa_{2}+\frac{\eta_{2}}{\chi_{2}}\Delta^{s}\DD\cdot\sss_{4}\Big),\CH^{\Delta^{s}}_{2}\Big\rangle\nonumber\\
&\qquad+\Big\langle\Big(\frac{1}{2}\Delta^{s}\WW\cdot\aaa_{3}+\frac{\eta_{1}}{\chi_{1}}\Delta^{s}\DD\cdot\sss_{5}\Big),\CH^{\Delta^{s}}_{1}\Big\rangle\nonumber\\
&\qquad+C\sum_{k=1}^{5}E^{\frac{k}{2}}_{s}(\Fp,\vv)D_{s}(\Fp,\vv).
\end{align}

To control the dissipative higher-order derivative term in (\ref{Ps}), we need the following lemma.

\begin{lemma}\label{lfs}
For any given $\Fp=(\nn_1,\nn_2,\nn_3)\in SO(3)$, there exists a constant $C>0$ such that
\begin{align}
\sum_{i=1}^{3}\frac{1}{\chi_{i}}\|\CH^{\Delta^{s}}_{i}\|_{L^{2}}^{2}\geq\frac{\gamma^{2}}{4\chi}\|\Delta^{s+1}\Fp\|_{L^{2}}^{2}-C\sum_{k=1}^{5}E^{\frac{k}{2}}_{s}(\Fp,\vv)D_{s}(\Fp,\vv), 
\end{align}
where $\gamma=\min\{\gamma_1,\gamma_2,\gamma_3\}$, $\chi=\max\{\chi_1,\chi_2,\chi_3\}$, and $\CH^{\Delta^s}_i(i=1,2,3)$ are defined by \eqref{CH}.
\end{lemma}
We relegate the proof of Lemma \ref{lfs} to the appendix so as not to distract from the main body of this article.

 From (\ref{Ps}) and Lemma \ref{lfs} we derive
\begin{align}\label{dv...}
&\frac{\ud}{\ud t}\sum_{i=1}^{3}\CE_{i}^{s}(\Fp)+\frac{\gamma^{2}}{4\chi}\|\Delta^{s+1}\Fp\|_{L^{2}}^{2}\nonumber\\   
&\quad\leq\Big\langle\Big(\frac{1}{2}\Delta^{s}\WW\cdot\aaa_{1}+\frac{\eta_{3}}{\chi_{3}}\Delta^{s}\DD\cdot\sss_{3}\Big),\CH^{\Delta^{s}}_{3}\Big\rangle\nonumber\\
&\qquad+\Big\langle\Big(\frac{1}{2}\Delta^{s}\WW\cdot\aaa_{2}+\frac{\eta_{2}}{\chi_{2}}\Delta^{s}\DD\cdot\sss_{4}\Big),\CH^{\Delta^{s}}_{2}\Big\rangle\nonumber\\
&\qquad+\Big\langle\Big(\frac{1}{2}\Delta^{s}\WW\cdot\aaa_{3}+\frac{\eta_{1}}{\chi_{1}}\Delta^{s}\DD\cdot\sss_{5}\Big),\CH^{\Delta^{s}}_{1}\Big\rangle\nonumber\\
&\qquad+C\sum_{k=1}^{5}E^{\frac{k}{2}}_{s}(\Fp,\vv)D_{s}(\Fp,\vv)\nonumber\\
&\quad=\Big\langle\frac{\eta_{3}}{\chi_{3}}(\CH^{\Delta^{s}}_{3})\sss_{3}+\frac{\eta_{2}}{\chi_{2}}(\CH^{\Delta^{s}}_{2})\sss_{4}+\frac{\eta_{1}}{\chi_{1}}(\CH^{\Delta^{s}}_{1})\sss_{5},\Delta^{s}\DD\Big\rangle\nonumber\\
&\qquad+\frac{1}{2}\Big\langle(\CH^{\Delta^{s}}_{3})\aaa_{1}+(\CH^{\Delta^{s}}_{2})\aaa_{2}+(\CH^{\Delta^{s}}_{1})\aaa_{3},\Delta^{s}\WW\Big\rangle\nonumber\\
&\qquad+C\sum_{k=1}^{5}E^{\frac{k}{2}}_{s}(\Fp,\vv)D_{s}(\Fp,\vv).
\end{align}

\subsection{The proof of Theorem \ref{flz}}
Armed with (\ref{dv..}), (\ref{dv.}) and (\ref{dv...}), we can get
\begin{align*}
 \frac{1}{2}\frac{\ud}{\ud t}E_{s}(\Fp,\vv)+D_{s}(\Fp,\vv)
 \leq C\sum_{k=1}^{5}E^{\frac{k}{2}}_{s}(\Fp,\vv)D_{s}(\Fp,\vv). 
\end{align*}
This implies that there exists $\varepsilon_{0}>0$ such that if $E_{s}(\Fp^{(0)},\vv^{(0)})\leq\varepsilon_{0},$ then
$$E_{s}(\Fp,\vv)(t)\leq E_{s}(\Fp^{(0)},\vv^{(0)})$$ 
for all $t\in[0,T]$.
Thus, using the blow-up criterion given by Theorem \ref{wcc}, the solution is global in time. As a consequence, the proof of Theorem \ref{flz} is finished.

\section{Appendix}
We here give the argument of Lemma \ref{lfs}. First of all, armed with the definitions of $\CH^{\Delta^s}_k(k=1,2,3)$ in (\ref{CH}) and Lemma \ref{pde}, together with the elementary inequality
\begin{align*}
    |a-b+c|^2\geq \frac{1}{2}|a|^2-4(|b|^2+|c|^2),~~a,b,c\in\mathbb{R},
\end{align*}
one has
\begin{align*}
\sum_{k=1}^{3}\frac{1}{\chi_{k}}\|\CH_{k}^{\Delta^{s}}\|_{L^{2}}^{2}=&\frac{1}{\chi_{1}}\Big\|\Delta^{s}(\ML_1\CF_{B{i}})-[\Delta^{s},\nn_{2}\cdot]\hh_{3}+[\Delta^{s},\nn_{3}\cdot]\hh_{2}\Big\|_{L^{2}}^{2}\\
&+\frac{1}{\chi_{2}}\Big\|\Delta^{s}(\ML_2\CF_{B{i}})-[\Delta^{s},\nn_{3}\cdot]\hh_{1}+[\Delta^{s},\nn_{1}\cdot]\hh_{3}\Big\|_{L^{2}}^{2}\\
&+\frac{1}{\chi_{3}}\Big\|\Delta^{s}(\ML_3\CF_{B{i}})-[\Delta^{s},\nn_{1}\cdot]\hh_{2}+[\Delta^{s},\nn_{2}\cdot]\hh_{1}\Big\|_{L^{2}}^{2}\\
\geq&\frac{1}{2\chi_{1}}\|\Delta^{s}(\ML_1\CF_{B{i}})\|_{L^{2}}^{2}-\frac{4}{\chi_{1}}\big(\|[\Delta^{s},\nn_{2}\cdot]\hh_{3}\|_{L^{2}}^{2}+\|[\Delta^{s},\nn_{3}\cdot]\hh_{2}\|_{L^{2}}^{2}\big)\\
&+\frac{1}{2\chi_{2}}\|\Delta^{s}(\ML_2\CF_{B{i}})\|_{L^{2}}^{2}-\frac{4}{\chi_{2}}\big(\|[\Delta^{s},\nn_{3}\cdot]\hh_{1}\|_{L^{2}}^{2}+\|[\Delta^{s},\nn_{1}\cdot]\hh_{3}\|_{L^{2}}^{2}\big)\\
&+\frac{1}{2\chi_{3}}\|\Delta^{s}(\ML_3\CF_{B{i}})\|_{L^{2}}^{2}-\frac{4}{\chi_{3}}\big(\|[\Delta^{s},\nn_{1}\cdot]\hh_{2}\|_{L^{2}}^{2}+\|[\Delta^{s},\nn_{2}\cdot]\hh_{1}\|_{L^{2}}^{2}\big)\\
\geq&\sum^3_{k=1}\frac{1}{2\chi_k}\|\Delta^s(\ML_k\CF_{Bi})\|^2_{L^2}-\frac{4}{\widetilde{\chi}}\sum^3_{i,j=1;i\neq j}\|[\Delta^{s},\nn_{i}\cdot]\hh_{j}\|_{L^{2}},
\end{align*}
where $\widetilde{\chi}=\min\{\chi_1,\chi_2,\chi_3\}$. Then, combining the above inequality with (\ref{Delta-s-nni-hhj}), we obtain
\begin{align}\label{C}
\sum_{k=1}^{3}\frac{1}{\chi_{k}}\|\CH_{k}^{\Delta^{s}}\|_{L^{2}}^{2}\geq&\sum_{k=1}^{3}\frac{1}{2\chi_{k}}\|\Delta^{s}(\ML_k\CF_{B{i}})\|_{L^{2}}^{2}-C\sum_{k=1}^{3}E^{\frac{k}{2}}_{s}(\Fp,\vv)D_{s}(\Fp,\vv).
\end{align}

To control the higher-order derivative terms involving the rotational differential operators $\ML_k(k=1,2,3)$ in (\ref{C}), we need to utilize the orthogonal decomposition on $T_{\Fp}SO(3)$. For any $A\in\mathbb{R}^{3\times3}$, taking $B=\Delta^{s+1}\Fp$ in $(\ref{decomposition})$, one can get that 
\begin{align}\label{decomposition 1}
A\cdot \Delta^{s+1}\Fp=\sum_{k=1}^{3}\frac{1}{|V_{k}|^2}(A\cdot V_{k})(\Delta^{s+1}\Fp\cdot V_k)+\sum_{k=1}^{6}\frac{1}{|W_{k}|^2}(A\cdot W_{k})(\Delta^{s+1}\Fp\cdot W_{k}),
\end{align}
 where $V_{k}(k=1,2,3)$ and $W_{k}(k=1,\cdots,6)$ are the orthogonal bases of the tangent space $T_{\Fp}SO(3)$ and its associated orthogonal complement space, respectively.

Taking 
\begin{align*}
A=\Delta^{s}\nabla\cdot\frac{\partial f_{B{i}}}{\partial(\nabla\Fp)}-\gamma\Delta^{s+1}\Fp,\quad \gamma=\min\{\gamma_1,\gamma_2,\gamma_3\}    
\end{align*}
into $(\ref{decomposition 1})$ and then integrating over $\mathbb{R}^{d}$, it follows that 
\begin{align}\label{Deltas-fBi-gamma}
&\int_{\mathbb{R}^{d}}\Big(\Delta^{s}\nabla\cdot\frac{\partial f_{B_{i}}}{\partial(\nabla\Fp)}-\gamma\Delta^{s+1}\Fp\Big)\cdot\Delta^{s+1}\Fp\ud\xx\nonumber\\
&\quad\leq\frac{1}{2}\int_{\mathbb{R}^{d}}\sum_{k=1}^{3}\Big[\Big(\Delta^{s}\nabla\cdot\frac{\partial f_{B_{i}}}{\partial(\nabla\Fp)}-\gamma\Delta^{s+1}\Fp\Big)\cdot V_{k}\Big](\Delta^{s+1}\cdot V_{k})\ud\xx\nonumber\\
&\qquad+\sum_{k=1}^{6}\underset{L_k}{\underbrace{\int_{\mathbb{R}^{d}}\frac{1}{|W_{k}|^{2}}\Big[\Big(\Delta^{s}\nabla\cdot\frac{\partial f_{B_{i}}}{\partial(\nabla\Fp)}-\gamma\Delta^{s+1}\Fp\Big)\cdot W_{k}\Big](\Delta^{s+1}\Fp\cdot W_{k})\ud\xx}}.
\end{align}
We next consider the estimates of the terms $L_{k}(k=1,2,3)$ in (\ref{Deltas-fBi-gamma}). Notice that from (\ref{hh-i3-definition}) and Lemma \ref{h-decomposition}, the lower-order derivative terms $\frac{\partial f_{Bi}}{\partial\nn_i}$ can be expressed by
\begin{align}\label{lower-order-pf}
     \frac{\partial f_{Bi}}{\partial\nn_i}=\sum^3_{j=1}k_{ij}(\nn_i\cdot\nabla\times\nn_j)(\nabla\times\nn_j),\quad i=1,2,3,
\end{align}
where the coefficients $k_{ij}$ are expressed by \eqref{ki-kij}.
Consequently, from (\ref{hh-i3-definition}) and the expression $W_1=(0,\nn_3,\nn_2)$, the term $L_1$ in (\ref{Deltas-fBi-gamma}) can be handled as
\begin{align}\label{L1-in}
L_{1}=&\frac{1}{2}\int_{\mathbb{R}^{d}}\Big[\nn_{3}\cdot\Big(\Delta^{s}\hh_{2}-\gamma\Delta^{s+1}\nn_{2}+\Delta^{s}\Big(\frac{\partial f_{B{i}}}{\partial\nn_{2}}\Big)\Big)\nonumber\\
&+\nn_{2}\cdot\Big(\Delta^{s}\hh_{3}-\gamma\Delta^{s+1}\nn_{3}+\Delta^{s}\Big(\frac{\partial f_{B{i}}}{\partial\nn_{3}}\Big)\Big)\Big]\big(\nn_{3}\cdot\Delta^{s+1}\nn_{2}+\nn_{2}\cdot\Delta^{s+1}\nn_{3}\big)\ud\xx\nonumber\\
=&\frac{1}{2}\int_{\mathbb{R}^{d}}\Big[\Delta^{s}(\nn_{3}\cdot\hh_{2})-[\Delta^{s},\nn_{3}\cdot]\hh_{2}-\gamma\nn_{3}\cdot\Delta^{s+1}\nn_{2}-\nn_{3}\cdot\Delta^{s}\Big(\frac{\partial f_{B{i}}}{\partial\nn_{2}}\Big)\nonumber\\
&+\Delta^{s}(\nn_{2}\cdot\hh_{3})-[\Delta^{s},\nn_{2}\cdot]\hh_{3}-\gamma\nn_{2}\cdot\Delta^{s+1}\nn_{3}+\nn_{2}\cdot\Delta^{s}\Big(\frac{\partial f_{B{i}}}{\partial\nn_{3}}\Big)\Big]\nonumber\\
&\quad\big(\nn_{3}\cdot\Delta^{s+1}\nn_{2}+\nn_{2}\cdot\Delta^{s+1}\nn_{3}\big)\ud\xx\nonumber\\
\leq&C\bigg(\|[\Delta^{s},\nn_{3}\cdot]\hh_{2}\|_{L^{2}}+\|[\Delta^{s},\nn_{2}\cdot]\hh_{3}\|_{L^{2}}+\Big\|\Delta^{s}\Big(\frac{\partial f_{B{i}}}{\partial\nn_{2}}\Big)\Big\|_{L^{2}}\nonumber\\
&+\|\Delta^{s}(\nn_{3}\cdot\hh_{2})\|_{L^{2}}+\|\Delta^{s}(\nn_{2}\cdot\hh_{3})\|_{L^{2}}+\Big\|\Delta^{s}\Big(\frac{\partial f_{B{i}}}{\partial\nn_{3}}\Big)\Big\|_{L^{2}}\nonumber\\
&+\big\|\nn_{3}\cdot\Delta^{s+1}\nn_{2}+\nn_{2}\cdot\Delta^{s+1}\nn_{3}\big\|_{L^{2}}\bigg)\big\|\nn_{3}\cdot\Delta^{s+1}\nn_{2}+\nn_{2}\cdot\Delta^{s+1}\nn_{3}\big\|_{L^{2}}.
\end{align}
For $i,j=2,3 (i\neq j)$, from Lemma \ref{pde}, we have
\begin{align*}
&\|\nn_{i}\cdot\Delta^{s+1}\nn_{j}+\nn_{j}\cdot\Delta^{s+1}\nn_{i}\|_{L^{2}}\\
&\quad=\|\nn_{i}\Delta^{s+1}\nn_{j}+\nn_{j}\Delta^{s+1}\nn_{i}-\Delta^{s+1}(\nn_{i}\cdot\nn_{j})\|_{L^{2}}\\
&\quad=\Big\|-\sum_{\substack{\alpha +\beta=2;\alpha,\beta \neq0}}^{2s+2}\CD^{\alpha}\nn_{i}\cdot\CD^{\beta}\nn_{j}\Big\|_{L^{2}}\\
&\quad\leq C\|\nabla^{2s}(\nabla\Fp\cdot\nabla\Fp)\|_{L^{2}} \\
&\quad\leq C\|\nabla\Fp\|_{L^{\infty}}\|\Delta^{s}\nabla\Fp\|_{L^{2}}.
\end{align*}
For $i=2,3$, from (\ref{lower-order-pf}) and Lemma \ref{pde} we get
\begin{align*}
    \Big\|\Delta^{s}\Big(\frac{\partial f_{B{i}}}{\partial\nn_{i}}\Big)\Big\|_{L^{2}}\leq&C\big(\|\Delta^{s}\nabla\Fp\|_{L^{2}}+\|\nabla\Fp\|_{L^{\infty}}\|\Delta^{s}\Fp\|_{L^{2}}\big).
\end{align*}
Then, combining the above inequalities with (\ref{L1-in}) together with Lemma \ref{pde} leads to
\begin{align*}
L_{1}\leq& C\Big(\|\Delta^{s}\Fp\|_{L^{2}}\|\hh_{j}\|_{L^{\infty}}+\|\nabla\Fp\|_{L^{\infty}}\|\nabla^{2s-1}\hh_{j}\|_{L^{2}}
+\|\nabla\Fp\|_{L^{\infty}}\|\Delta^{s}\nabla\Fp\|_{L^{2}}\\
&+\|\Delta^{s}\hh_{j}\|_{L^{2}}+\|\Delta^s\nabla\Fp\|_{L^2}\Big)\|\nabla\Fp\|_{L^{\infty}}\|\Delta^{s}\nabla\Fp\|_{L^{2}}\\
\leq&C\sum_{k=1}^{5}E^{\frac{k}{2}}_{s}(\Fp,\vv)D_{s}(\Fp,\vv).
\end{align*}
Similar to the estimate of $L_1$, there also has
\begin{align*}
L_{2}+L_3\leq C\sum_{k=1}^{5}E^{\frac{k}{2}}_{s}(\Fp,\vv)D_{s}(\Fp,\vv).
\end{align*}

We now turn to the estimates of the terms $L_k(k=4,5,6)$ in (\ref{Deltas-fBi-gamma}).  Using the expression $W_4=(\nn_1,0,0)$ and the definition (\ref{hh-i3-definition}), the term $L_4$ can be calculated as
\begin{align*}
 L_{4}=&\int_{\mathbb{R}^{d}}\Big[\nn_{1}\cdot\Delta^{s}\hh_{1}-\gamma\nn_1\cdot\Delta^{s+1}\nn_{1}+\nn_{1}\cdot\Delta^{s}\Big(\frac{\partial f_{B_{i}}}{\partial\nn_{1}}\Big)\Big](\nn_{1}\cdot\Delta^{s+1}\nn_{1})\ud\xx\\
\leq&C\Big(\|\Delta^{s}(\nn_{1}\cdot\hh_{1})\|_{L^{2}}+\|\Delta^s,\nn_1\cdot]\hh_1\|_{L^2}+\|\nn_{1}\cdot\Delta^{s+1}\nn_{1}\|_{L^{2}}\\
&+\Big\|\Delta^{s}\Big(\frac{\partial f_{B{i}}}{\partial\nn_{1}}\Big)\Big\|_{L^{2}}\Big)\|\nn_{1}\cdot\Delta^{s+1}\nn_{1}\|_{L^{2}}, 
\end{align*}
where
\begin{align*}    \|\nn_{1}\cdot\Delta^{s+1}\nn_{1}\|_{L^{2}}=&\Big\|\nn_{1}\cdot\Delta^{s+1}\nn_{1}-\frac{1}{2}\Delta^{s+1}(\nn_{1}\cdot\nn_{1})\Big\|_{L^{2}}\\
    =&\Big\|-\sum_{\substack{\alpha +\beta=2;\alpha,\beta \neq0}}^{2s+2}\CD^{\alpha}\nn_{1}\cdot\CD^{\beta}\nn_{1}\Big\|_{L^{2}}\\
    \leq& C\|\Delta^{s}(\nabla\Fp\cdot\nabla\Fp)\|_{L^{2}}\\
    \leq& C\|\nabla\Fp\|_{L^{\infty}}\|\Delta^s\nabla\Fp\|_{L^2}.
\end{align*}
Then, armed with Lemma \ref{pde} we can get that
\begin{align*}
    L_4\leq C\sum_{k=1}^{5}E^{\frac{k}{2}}_{s}(\Fp,\vv)D_{s}(\Fp,\vv).
\end{align*}
Again, the terms $L_5,L_6$ enjoy the same estimates as the term $L_4$. Summing up the terms from $L_1$ to $L_6$ yields
\begin{align}\label{L1-6}
    \sum_{k=1}^{6}L_{k}\leq&C\sum_{k=1}^{5}E^{\frac{k}{2}}_{s}(\Fp,\vv)D_{s}(\Fp,\vv).
\end{align}
Thus, from (\ref{Deltas-fBi-gamma}) and (\ref{L1-6}) we obtain
\begin{align}\label{Delta-s-partial-fBi}
&\int_{\mathbb{R}^{d}}\Big(\Delta^{s}\nabla\cdot\frac{\partial f_{B_{i}}}{\partial(\nabla\Fp)}-\gamma\Delta^{s+1}\Fp\Big)\cdot\Delta^{s+1}\Fp\ud\xx\nonumber\\
&\quad\leq\frac{1}{2}\int_{\mathbb{R}^{d}}\sum_{k=1}^{3}\Big[\Big(\Delta^{s}\nabla\cdot\frac{\partial f_{B_{i}}}{\partial(\nabla\Fp)}-\gamma\Delta^{s+1}\Fp\Big)\cdot V_{k}\Big](\Delta^{s+1}\cdot V_{k})\ud\xx\nonumber\\
&\qquad+C\sum_{k=1}^{5}E^{\frac{k}{2}}_{s}(\Fp,\vv)D_{s}(\Fp,\vv).
\end{align}

On the other hand, taking $A=\Delta^{s+1}\Fp$ in $(\ref{decomposition 1}),$ and then integrating over $\mathbb{R}^{d}$, we infer that
\begin{align}\label{Delta-s+1-Fp}
\|\Delta^{s+1}\Fp\|_{L^{2}}^{2}=&\frac{1}{2}\sum_{k=1}^{3}\|\Delta^{s+1}\Fp\cdot V_{k}\|_{L^{2}}^{2}+\sum_{k=1}^{6}\frac{1}{|W_{k}|^{2}} \|\Delta^{s+1}\Fp\cdot W_{k}\|_{L^{2}}^{2}\nonumber\\
\leq&\frac{1}{2}\sum_{k=1}^{3}\|\Delta^{s+1}\Fp\cdot V_{k}\|_{L^{2}}^{2}+C\sum^3_{i,j=1}\|\nn_i\cdot\Delta^{s+1}\nn_j\|^2_{L^2}\nonumber\\
\leq&\frac{1}{2}\sum_{k=1}^{3}\|\Delta^{s+1}\Fp\cdot V_{k}\|_{L^{2}}^{2}+CE_{s}(\Fp,\vv)D_{s}(\Fp,\vv).
\end{align}
Consequently, by means of (\ref{C}) and the definitions of $\CH^{\Delta^s}_k$ and $\ML_k(k=1,2,3)$, we derive that
\begin{align}\label{CH-k-Delta-s}
\sum_{k=1}^{3}\frac{1}{\chi_{k}}\|\CH_{k}^{\Delta^{s}}\|_{L^{2}}^{2}\geq&\sum_{k=1}^{3}\frac{1}{2\chi_{k}}\|\Delta^{s}(\ML_k\CF_{B{i}})\|_{L^{2}}^{2}-C\sum_{k=1}^{3}E^{\frac{k}{2}}_{s}(\Fp,\vv)D_{s}(\Fp,\vv)\nonumber\\
=&\sum^3_{k=1}\frac{1}{2\chi_k}\Big\|\Delta^s\Big(V_k\cdot\frac{\delta\CF_{Bi}}{\delta\Fp}\Big)\Big\|^2_{L^2}-C\sum_{k=1}^{3}E^{\frac{k}{2}}_{s}(\Fp,\vv)D_{s}(\Fp,\vv)\nonumber\\
=&\sum_{k=1}^{3}\frac{1}{2\chi_{k}}\Big\|\Delta^{s}\Big(V_{k}\cdot\Big(\nabla\cdot\frac{\partial f
_{B{i}}}{\partial(\nabla\Fp)}\Big)\Big)-\Delta^{s}\Big(V_{k}\cdot\frac{\partial f
_{B{i}}}{\partial\Fp}\Big)\Big\|_{L^{2}}^{2}\nonumber\\
&-C\sum_{k=1}^{3}E^{\frac{k}{2}}_{s}(\Fp,\vv)D_{s}(\Fp,\vv)\nonumber\\
\geq&\sum_{k=1}^{3}\frac{1}{4\chi_{k}}\Big\|\Delta^{s}\Big(V_{k}\cdot\Big(\nabla\cdot\frac{\partial f
_{B{i}}}{\partial(\nabla\Fp)}\Big)\Big)\Big\|_{L^{2}}^{2}-\sum_{k=1}^{3}\frac{2}{\chi_{k}}\Big\|\Delta^{s}\Big(V_{k}\cdot\frac{\partial f
_{B{i}}}{\partial\Fp}\Big)\Big\|_{L^{2}}^{2}\nonumber\\
&-C\sum_{k=1}^{3}E^{\frac{k}{2}}_{s}(\Fp,\vv)D_{s}(\Fp,\vv).
\end{align}
As for the second term on the right side of the above inequality, by the definitions of $V_k(k=1,2,3)$ and (\ref{lower-order-pf}), together with Lemma \ref{pde}, we infer that
\begin{align*}
&\sum^3_{k=1}\frac{2}{\chi_k}\Big\|\Delta^{s}\Big(V_{k}\cdot\frac{\partial f
_{B{i}}}{\partial\Fp}\Big)\Big\|_{L^{2}}
\leq C\sum^3_{i,j=1;i\neq j}\Big\|\Delta^{s}\Big(\nn_{i}\cdot\frac{\partial f_{B{i}}}{\partial\nn_{j}}\Big)\Big\|_{L^{2}}\\
&\quad\leq C\sum^3_{i,j=1;i\neq j}\Big(\Big\|\frac{\partial f_{B{i}}}{\partial\nn_{j}}\cdot\Delta^{s}\nn_{i}\Big\|_{L^{2}}+\Big\|\nn_{i}\cdot\Delta^s\frac{\partial f_{B{i}}}{\partial\nn_{j}}\Big\|_{L^{2}}+\Big\|\sum_{\substack{\alpha +\beta=2\\\alpha,\beta \neq0}}^{2s}\CD^{\alpha}\nn_{i}\cdot\CD^{\beta}\frac{\partial f_{B{i}}}{\partial\nn_{j}}\Big\|_{L^{2}}\Big)\\
&\quad\leq C\Big(\|\nabla\Fp\|_{L^{\infty}}^{2}\|\Delta^{s}\Fp\|_{L^{2}}+\|\nabla\Fp\|_{L^{\infty}}\|\Delta^{s}\nabla\Fp\|_{L^{2}}
+\Big\|\nabla^{2s-2}\Big(\nabla\Fp\cdot\nabla\frac{\partial f_{B{i}}}{\partial\nn_{j}}\Big)\Big\|_{{L}^{2}}\Big)\\
&\quad\leq C\Big(\|\nabla\Fp\|_{L^{\infty}}^{2}\|\Delta^{s}\Fp\|_{L^{2}}+\|\nabla\Fp\|_{L^{\infty}}\|\Delta^{s}\nabla\Fp\|_{L^{2}}+\|\nabla\Fp\|_{L^{\infty}}\Big\|\nabla^{2s-1}\frac{\partial f_{B{i}}}{\partial\nn_{j}}\Big\|_{L^{2}}\\
&\qquad+\Big\|\nabla\frac{\partial f_{B{i}}}{\partial\nn_{j}}\Big\|_{L^{\infty}}\|\nabla^{2s-1}\Fp\|_{L^{2}}\Big)\\
&\quad\leq C\sum^3_{k=1}E^{\frac{k}{2}}_{s}(\Fp,\vv)D^{\frac{1}{2}}_{s}(\Fp,\vv).
\end{align*}
Noticing that from (\ref{hh-i3-definition}) we have
\begin{align}\label{higher-order-terms}
 \nabla\cdot\frac{\partial f
_{B_{i}}}{\partial(\nabla\nn_i)}=\gamma_i\Delta\nn_i+k_i\nabla{\rm div}\nn_i-\sum^3_{j=1}k_{ji}\nabla\times(\nabla\times\nn_i\cdot(\nn_j\otimes\nn_j)),\quad i=1,2,3.   
\end{align}
Similar to the above derivation, from (\ref{higher-order-terms}) and Lemma \ref{pde} one also has
\begin{align*}
&\Big\|\Delta^{s}\Big(V_{k}\cdot\nabla\cdot\frac{\partial f
_{B{i}}}{\partial(\nabla\Fp)}\Big)\Big\|_{L^{2}}^{2}\\
&\quad\geq \frac{1}{2}\Big\|V_{k}\cdot\Big(\Delta^{s}\nabla\cdot\frac{\partial f
_{B{i}}}{\partial(\nabla\Fp)}\Big)\Big\|_{L^{2}}^{2}-\Big\|[\Delta^s,V_k\cdot]\Big(\nabla\cdot\frac{\partial f_{Bi}}{\partial(\nabla\Fp)}\Big)\Big\|^2_{L^2}\\
&\quad\geq \frac{1}{2}\Big\|V_{k}\cdot\Big(\Delta^{s}\nabla\cdot\frac{\partial f
_{B{i}}}{\partial(\nabla\Fp)}\Big)\Big\|_{L^{2}}^{2}-C\Big(E_{s}(\Fp,\vv)+E^{\frac{3}{2}}_{s}(\Fp,\vv)+E^2_{s}(\Fp,\vv)\Big)D_{s}(\Fp,\vv).
\end{align*}
Hence, by means of (\ref{CH-k-Delta-s}) and the above estimates, we derive that
\begin{align}\label{CH-s-chi-k}
\sum_{k=1}^{3}\frac{1}{\chi_{k}}\|\CH_{k}^{\Delta^{s}}\|_{L^{2}}^{2}\geq&\sum_{k=1}^{3}\frac{1}{4\chi_{k}}\Big\|V_{k}\cdot\Big(\Delta^{s}\nabla\cdot\frac{\partial f
_{B{i}}}{\partial(\nabla\Fp)}\Big)\Big\|_{L^{2}}^{2}-C\sum_{k=1}^{4}E^{\frac{k}{2}}_{s}(\Fp,\vv)D_{s}(\Fp,\vv)\nonumber\\
\geq&\frac{\gamma}{4\chi}\int_{\mathbb{R}^d}\sum^3_{k=1}\Big[\Big(\Delta^s\nabla\cdot\frac{\partial f_{Bi}}{\partial(\nabla\Fp)}-\gamma\Delta^{s+1}\Fp\Big)\cdot V_k\Big](\Delta^{s+1}\Fp\cdot V_k)\ud\xx\nonumber\\
&+\frac{\gamma^2}{8\chi}\int_{\mathbb{R}^d}\sum^3_{k=1}(\Delta^{s+1}\Fp\cdot V_k)^2\ud \xx-C\sum_{k=1}^{5}E^{\frac{k}{2}}_{s}(\Fp,\vv)D_{s}(\Fp,\vv).
\end{align}

To deal with the first term on the right side of the last inequality in (\ref{CH-s-chi-k}), we prepare some necessary estimates.
Using Lemma \ref{pde} and integrating by parts, one can give 
\begin{align}\label{higher-term-1}
&-\langle\Delta^{s}\nabla\times(\nabla\times\nn_{i}\cdot(\nn_{j}\otimes\nn_j)),\Delta^{s+1}\nn_{i}\rangle\nonumber\\
&\quad=-\langle\Delta^{s}(\nabla\times\nn_{i}\cdot(\nn_{j}\otimes\nn_j)),\Delta^{s+1}(\nabla\times\nn_{i})\rangle\nonumber\\
&\quad=\langle\nabla\Delta^{s}(\nabla\times\nn_{i}\cdot(\nn_{j}\otimes\nn_j)),\nabla\Delta^{s}(\nabla\times\nn_{i})\rangle\nonumber\\
&\quad=\langle\nabla\Delta^{s}(\nabla\times\nn_{i})\cdot(\nn_{j}\otimes\nn_j),\nabla\Delta^{s}(\nabla\times\nn_{i})\rangle\nonumber\\
&\qquad+\langle[\nabla\Delta^{s},(\nn_{j}\otimes\nn_j)\cdot](\nabla\times\nn_{i}),\nabla\Delta^{s}(\nabla\times\nn_{i})\rangle\nonumber\\
&\quad\geq \|\nabla\Delta^{s}(\nabla\times\nn_{i})\cdot\nn_{j}\|_{L^{2}}^{2}-\|[\nabla\Delta^{s},(\nn_{j}\otimes\nn_j)\cdot](\nabla\times\nn_{i})\|_{L^{2}}\|\nabla\Delta^s(\nabla\times\nn_i)\|_{L^2}\nonumber\\
&\quad\geq\|\nabla\Delta^{s}(\nabla\times\nn_{i})\cdot\nn_{j}\|_{L^{2}}^{2}-C\big(E^{\frac{1}{2}}_{s}(\Fp,\vv)+E_{s}(\Fp,\vv)\big)D_{s}(\Fp,\vv),
\end{align}
where we have used the following estimate:
\begin{align*}
    &\|[\nabla\Delta^{s},(\nn_{j}\otimes\nn_j)\cdot](\nabla\times\nn_{i})\|_{L^{2}}\\
    &\quad\leq C\Big(\|\nabla^{2s+1}(\nn_{j}^{2})\|_{L^{2}}\|\nabla\nn_{i}\|_{L^{\infty}}+\|\nabla^{2}(\nn_{j}^{2})\|_{L^{\infty}}\|\nabla^{2s}\nn_{i}\|_{L^{2}}\Big)\\
    &\quad\leq C\Big(\|\nabla\Fp\|_{L^{\infty}}\|\nabla^{2s+1}\Fp\|_{L^{2}}+(\|\nabla^{2}\Fp\|_{L^{\infty}}+\|\nabla{\Fp}\|^{2}_{L^{\infty}})\|\nabla^{2s}\Fp\|_{L^{2}}\Big)\\
    &\quad\leq C\big(E^{\frac{1}{2}}_{s}(\Fp,\vv)+E_{s}(\Fp,\vv)\big)D^{\frac{1}{2}}_{s}(\Fp,\vv).
\end{align*}
Again, applying Lemma \ref{pde} and integrating by parts, we obtain
\begin{align}\label{higher-term-2}
&-\big\langle\Delta^{s}\big((\nn_{i}\cdot\nabla\times\nn_{j})(\nabla\times\nn_{j})\big),\Delta^{s+1}\nn_{i}\big\rangle\nonumber\\
&\quad=-\big\langle\Delta^{s}(\nn_{i}\cdot\nabla\times\nn_{j})(\nabla\times\nn_{j}),\Delta^{s+1}\nn_{i}\big\rangle-\big\langle[\Delta^{s},\nabla\times\nn_{j}](\nn_{i}\cdot\nabla\times\nn_{j}),\Delta^{s+1}\nn_{i}\big\rangle\nonumber\\
&\quad=\big\langle\nabla\Delta^{s}(\nn_{i}\cdot\nabla\times\nn_{j})(\nabla\times\nn_{j}),\nabla\Delta^{s}\nn_{i}\big\rangle+\big\langle\Delta^{s}(\nn_{i}\cdot\nabla\times\nn_{j})\nabla(\nabla\times\nn_{j}),\nabla\Delta^{s}\nn_{i}\big\rangle\nonumber\\
&\qquad-\big\langle[\Delta^{s},\nabla\times\nn_{j}](\nn_{i}\cdot\nabla\times\nn_{j}),\Delta^{s+1}\nn_{i}\big\rangle\nonumber\\
&\quad=\big\langle\nabla\Delta^{s}\nn_{i}\cdot(\nabla\times\nn_{j}),\nabla\Delta^{s}\nn_{i}\cdot(\nabla\times\nn_{j})\big\rangle+\big\langle[\nabla\Delta^{s},\nabla\times\nn_{j}](\nn_{i}\cdot\nabla\times\nn_{j}),\nabla\Delta^{s}\nn_{i}\big\rangle\nonumber\\
&\qquad+\big\langle\Delta^{s}(\nn_{i}\cdot\nabla\times\nn_{j})\nabla(\nabla\times\nn_{j}),\nabla\Delta^{s}\nn_{i}\big\rangle-\big\langle[\Delta^{s},\nabla\times\nn_{j}](\nn_{i}\cdot\nabla\times\nn_{j}),\Delta^{s+1}\nn_{i}\big\rangle\nonumber\\
&\quad\geq\|\nabla\Delta^{s}\nn_{i}\cdot(\nabla\times\nn_{j})\|_{L^{2}}^{2}-C\big(E^{\frac{1}{2}}_{s}(\Fp,\vv)+E_{s}(\Fp,\vv)\big)D_{s}(\Fp,\vv),
\end{align}
where we have used the following estimates:
\begin{align*}
    &\|[\nabla\Delta^{s},\nabla\times\nn_{j}](\nn_{i}\cdot\nabla\times\nn_{j})\|_{L^{2}}\\
    &\quad\leq C\Big(\|\nabla^{2s+2}\nn_{j}\|_{L^{2}}\|\nn_{i}\cdot\nabla\times\nn_{j}\|_{L^{\infty}}+\|\nabla^{2}\nn_{j}\|_{L^{\infty}}\|\nabla^{2s+1}\nn_{j}\|_{L^{2}}\Big)\\
    &\quad\leq CE^{\frac{1}{2}}_{s}(\Fp,\vv)D^{\frac{1}{2}}_{s}(\Fp,\vv),\\
    &\|\Delta^{s}(\nn_{i}\cdot\nabla\times\nn_{j})\nabla(\nabla\times\nn_{j})\|_{L^{2}}\\
    &\quad\leq C\|\nabla^{2}\Fp\|_{L^{\infty}}\Big(\|\nn_{i}\|_{L^{\infty}}\|\nabla^{2s+1}\nn_{j}\|_{L^{2}}+\|\nabla\nn_{j}\|_{L^{\infty}}\|\nabla^{2s}\nn_{i}\|_{L^{2}}\Big)\\
    &\quad\leq C\big(E^{\frac{1}{2}}_{s}(\vv,\Fp)+E_{s}(\Fp,\vv)\big)D^{\frac{1}{2}}_{s}(\Fp,\vv),\\
    &\|[\Delta^{s},\nabla\times\nn_{j}](\nn_{i}\cdot\nabla\times\nn_{j})\|_{L^{2}}\\
    &\quad\leq C\Big(\|\nabla^{2s+1}\nn_{j}\|_{L^{2}}\|\nn_{i}\cdot\nabla\times\nn_{j}\|_{L^{\infty}}+\|\nabla^{2}\nn_{j}\|_{L^{\infty}}\|\nabla^{2s-1}(\nn_{i}\cdot\nabla\times\nn_{j})\|_{L^{2}}\Big)\\
    &\quad\leq C\Big(\|\nabla\Fp\|_{L^{\infty}}\|\nabla^{2s+1}\Fp\|_{L^{2}}+\|\nabla^{2}\Fp\|_{L^{\infty}}\big(\|\nabla^{2s}\nn_{j}\|_{L^{2}}+\|\nabla\nn_{j}\|_{L^{\infty}}\|\nabla^{2s-1}\nn_{i}\|_{L^{2}}\big)\Big)\\
    &\quad\leq C\big(E^{\frac{1}{2}}_{s}(\Fp,\vv)+E_{s}(\Fp,\vv)\big)D^{\frac{1}{2}}_{s}(\Fp,\vv).
\end{align*}
A direct calculation leads to
\begin{align}\label{higher-term-3}
\langle\Delta^{s}\nabla\text{div}\nn_{i},\Delta^{s+1}\nn_{i}\rangle=-\langle\Delta^{s}\text{div}\nn_{i},\Delta^{s+1}\text{div}\nn_{i}\rangle
=\|\nabla\Delta^{s}\text{div}\nn_{i}\|^2_{L^2}.
\end{align}
Then, combining (\ref{higher-order-terms}) with (\ref{higher-term-1})--(\ref{higher-term-3}), we infer that
\begin{align}\label{Delta-s-gamma-s+1}
&\int_{\mathbb{R}^{d}}\Big(\Delta^{s}\nabla\cdot\frac{\partial f_{B{i}}}{\partial(\nabla\Fp)}-\gamma\Delta^{s+1}\Fp\Big)\cdot\Delta^{s+1}\Fp\ud\xx\nonumber\\
&\quad\geq\sum_{k=1}^{3}\int_{\mathbb{R}^{d}}\big(k_{i}|\nabla\Delta^{s}\text{div}\nn_{i}|^{2}+\sum_{j=1}^{3}k_{ji}|\Delta^{s}\nabla\times(\nabla\nn_{i}\cdot\nn_{j})|^{2}\big)\ud\xx\nonumber\\
&\qquad-C\big(E_{s}(\Fp,\vv)^{\frac{1}{2}}+E_{s}(\Fp,\vv)\big)D_{s}(\Fp,\vv).
\end{align}
Consequently, plugging (\ref{Delta-s+1-Fp}) and (\ref{Delta-s-gamma-s+1}) into (\ref{CH-s-chi-k}) yields
\begin{align*}
\sum_{k=1}^{3}\frac{1}{\chi_{k}}\|\CH_{k}^{\Delta^{s}}\|_{L^{2}}^{2}
\geq&\sum_{k=1}^{3}\int_{\mathbb{R}^{d}}\big(k_{i}|\nabla\Delta^{s}\text{div}\nn_{i}|^{2}+\sum_{j=1}^{3}k_{ji}|\Delta^{s}\nabla\times(\nabla\nn_{i}\cdot\nn{j})|^{2}\big)\ud\xx\\
&+\frac{\gamma^{2}}{4\chi}\|\Delta^{s+1}\Fp\|_{L^{2}}^{2}-C\sum_{k=1}^{5}E^{\frac{k}{2}}_{s}(\Fp,\vv)D_{s}(\Fp,\vv)\\
\geq&\frac{\gamma^{2}}{4\chi}\|\Delta^{s+1}\Fp\|_{L^{2}}^{2}-C\sum_{k=1}^{5}E^{\frac{k}{2}}_{s}(\Fp,\vv)D_{s}(\Fp,\vv).  
\end{align*}
Lemma \ref{lfs} is thus proved.

\bigskip

\end{document}